\documentclass[11pt]{amsart}

\usepackage{a4wide,amsmath,amsfonts,amssymb,amsthm,mathrsfs,bbm}
\usepackage[hyperfootnotes=false]{hyperref}
\usepackage[dvipsnames]{xcolor}

\allowdisplaybreaks[2]

\numberwithin{equation}{section}

\newtheorem{theorem}{Theorem}[section]
\newtheorem{lemma}[theorem]{Lemma}

\newtheorem{remark}[theorem]{Remark}

\newtheorem{definition}[theorem]{Definition}
\newtheorem{assumption}[theorem]{Assumption}

\renewcommand{\d}{\mathrm{d}}
\newcommand{\dd}{\,\mathrm{d}}
\newcommand{\R}{\mathbb{R}}
\newcommand{\bR}{\mathbb{R}}
\newcommand{\N}{\mathbb{N}}
\newcommand{\bN}{\mathbb{N}}
\newcommand{\E}{\mathbb{E}}
\newcommand{\bE}{\mathbb{E}}
\renewcommand{\P}{\mathbb{P}}
\renewcommand{\epsilon}{\varepsilon}
\newcommand{\bP}{\mathbb{P}}
\def\L{{\mathcal{L}}}

\definecolor{ob}{RGB}{72,145,220}  
\definecolor{or}{RGB}{190,15,52}   
\definecolor{og}{RGB}{105,145,59}  
\definecolor{oo}{RGB}{255,128,0}   
\definecolor{oy}{RGB}{245,207,71}  

\title[Stability results for distribution-dependent stochastic Volterra equations]{Stability results for distribution-dependent stochastic Volterra equations}

\author[Bergerhausen]{Martin Bergerhausen}
\address{Martin Bergerhausen, University of Mannheim, Germany}
\email{martin.bergerhausen@uni-mannheim.de}

\author[Pr{\"o}mel]{David J. Pr{\"o}mel}
\address{David J. Pr{\"o}mel, University of Mannheim, Germany}
\email{proemel@uni-mannheim.de}

\date{\today}

\begin{document}

\begin{abstract}
  We investigate stability properties of distribution-dependent stochastic Volterra equations with respect to changes in the coefficients, the Volterra kernels, and the initial condition. Under Lipschitz continuity assumptions on the coefficients, we first derive quantitative stability estimates for strong solutions with explicit error bounds. We then prove a general convergence theorem for strong solutions under substantially weaker assumptions, replacing Lipschitz continuity by a continuity assumption together with uniform linear growth of the approximating sequence. Finally, we study the associated distribution-dependent Volterra local martingale problem and prove the stability of its solutions under convergence of the coefficients, kernels, and initial distributions.
\end{abstract}

\maketitle




\section{Introduction}

Distribution-dependent stochastic Volterra equations extend classical stochastic differential equations (SDEs) by combining two distinct features: memory effects, incorporated through Volterra kernels, and mean-field interactions, arising from the dependence of the coefficients on the marginal law of the solution. They therefore encompass both stochastic Volterra equations (SVEs) and McKean--Vlasov stochastic differential equations as particular cases. Stochastic Volterra equations have been studied since the pioneering works of Berger and Mizel \cite{Berger1980a,Berger1980b} and have attracted renewed attention in recent years, driven, among other things, by applications in mathematical finance; see, for instance, \cite{Jaisson2016,ElEuch2019,Mishura2024}. On the other hand, McKean--Vlasov SDEs, also known as distribution-dependent or mean-field SDEs, were introduced in the seminal works of Kac~\cite{Kac56}, McKean~\cite{McKean66} and Vlasov~\cite{Vlasov1968} and provide a well-established framework for modelling mean-field interactions and the dynamics of large systems of interacting particles; we refer to \cite{Carmona2018,Carmona2018b,Chaintron2022,Chaintron2022b} for a comprehensive account.

In the present paper we study the stability of distribution-dependent stochastic Volterra equations (DDSVEs) of the form
\begin{equation}\label{eq:intro}
  X_t = X_0 +\int_0^t K_{b}(s,t) b(s,X_s,\mathcal{L}(X_s))\dd s+\int_0^t K_{\sigma}(s,t)\sigma(s,X_s,\mathcal{L}(X_s))\dd B_s,\quad t\in [0,T],
\end{equation}
with respect to perturbations in the coefficients $b\colon[0,T]\times\R^d\times\mathcal{P}_{\eta}(\R^d)\to\R^d$ and $\sigma\colon[0,T]\times\R^d\times\mathcal{P}_{\eta}(\R^{d})\to\R^{d\times m}$, the (Volterra) kernels $K_b, K_\sigma\colon \Delta_T\to \R$, and the initial condition $X_0$, where $(B_t)_{t\in [0,T]}$ denotes a multi-dimensional Brownian motion. Throughout this work, we impose only fairly mild integrability and regularity assumptions on the kernels $K_b$ and $K_\sigma$, which, in particular, allow for singular kernels; cf. Subsection~\ref{subsec: kernel assumptions}.

A substantial well-posedness theory for DDSVEs has emerged in recent years, see \cite{Shi2013,Promel_MF,Liu2023,Jie2024,Kalinin2026,Bergerhausen2026}. This includes results on the existence and uniqueness of strong solutions, the existence of weak solutions, and the formulation of an associated distribution-dependent Volterra local martingale problem. By contrast, systematic results concerning the stability of DDSVEs under perturbations of their defining parameters remain comparatively scarce.

Stability is a fundamental component of the theory of stochastic differential equations. It quantifies the sensitivity of solutions to perturbations of the initial data and model coefficients and, at a more qualitative level, identifies conditions under which approximating stochastic systems converge to a prescribed limiting dynamics. In particular, such results are indispensable for the analysis of approximation schemes, numerical methods and statistical inference, as well as closely connected with existence proofs based on compactness and approximation. For ordinary and McKean--Vlasov stochastic differential equations, a variety of stability principles forms part of the classical theory; see, for example, the textbooks \cite{Stroock1979,Jacod2003,Carmona2018,Carmona2018b}. In the Volterra setting, however, the situation is more delicate since the solution of an SVE is generally neither Markovian nor a semimartingale. For classical SVEs, first stability results have been developed as a key ingredient in weak existence arguments; see \cite{AbiJaber2021,Promel2023_weak,AbiJaber2025}. Quantitative stability estimates also arise naturally in approximation and learning problems for SVEs; see, for instance, \cite{Bergerhausen2025_2}. The distribution-dependent setting introduces an additional layer of complexity through the nonlinear dependence of the dynamics on the marginal law of the solution.

Our first contribution is a quantitative stability estimate for strong solutions under global Lipschitz and linear growth assumptions on the coefficients, see Theorem~\ref{theorem:stability_MFSVE}. More precisely, we control the distance between two solutions in terms of the distance between their initial conditions, their drift and diffusion coefficients, and their Volterra kernels. The resulting estimate extends the corresponding result for ordinary SVEs in \cite[Proposition~2.1]{Bergerhausen2025_2} in two directions: the coefficients may depend on the marginal distribution of the solution, and the kernels are not required to be of convolution type. In particular, the estimate provides quantitative continuous dependence of DDSVEs on all components entering the equation.

Our second contribution is a considerably more general convergence theorem for strong  solutions which no longer requires Lipschitz continuity of the involved coefficients, see Theorem~\ref{theorem:approximation}. Assuming that the considered DDSVEs admit uniform strong solutions and that the approximating coefficients satisfy a uniform linear growth condition, we prove convergence of the corresponding solution provided the coefficients and kernels converge in suitable topologies. More precisely, the convergence holds in the expected uniform norm and uniformly over compact sets of initial values. The proof does not rely on a direct Lipschitz estimate. Instead, it combines uniform moment and regularity bounds with tightness, Skorokhod representation, and identification of subsequential limits through the distribution-dependent Volterra martingale problem. This result therefore provides a qualitative stability principle applicable well beyond the perturbative Lipschitz regime.

Our third contribution concerns stability at the level of the local martingale problem associated to the distribution-dependent stochastic Volterra equation~\eqref{eq:intro}, as introduced in \cite{Bergerhausen2026}. We consider a sequence of martingale problems whose initial distributions, coefficients, and Volterra kernels converge to corresponding limiting objects and prove weak convergence of their solutions, provided that the limiting martingale problem is unique in law, see Theorem~\ref{theorem:approximation_weak}. In view of the equivalence between the Volterra martingale problem and weak solutions established in \cite{Bergerhausen2026}, it yields a natural weak stability principle for DDSVEs. A central ingredient in the proof is the continuity, under simultaneous perturbation of the kernels and integrators, of the Volterra-type integral operators arising in the martingale formulation. The corresponding auxiliary convergence results are established in the appendix and may be useful beyond the present setting.

\medskip
\noindent \textbf{Organization of the paper:} In Section~\ref{sec:VMP} we introduce distribution-dependent stochastic Volterra equations together with their associated distribution-dependent Volterra local martingale problems. Section~\ref{subsec:strong_results} contains the stability theorems for strong solutions, that is, a quantitative estimate under Lipschitz assumptions and a general convergence theorem under continuity and growth assumptions. Section~\ref{subsec:weak_results} establishes the stability of solutions to the distribution-dependent Volterra local martingale problem. The appendix~\ref{sec: appendix} collects auxiliary results regarding stochastic Volterra-type integrals.

\medskip
\noindent\textbf{Acknowledgments:} D. J. Pr\"omel gratefully acknowledges his affiliation with the Department of Mathematics at King’s College London, United Kingdom.

\section{Distribution-dependent stochastic Volterra equations}\label{sec:VMP}

Let $T\in (0,\infty)$, $d,m\in\N$, and let $(\Omega,\mathcal{F},(\mathcal{F}_t)_{t\in [0,T]},\mathbb{P})$ be a filtered probability space, which satisfies the usual conditions of completeness and right-continuity. Suppose $B=(B_t)_{t\in [0,T]}$ is an $m$-dimensional Brownian motion with respect to $(\mathcal{F}_t)_{t\in [0,T]}$. The law of a random variable~$X$ is denoted by $\mathcal{L}(X)$ and, if $(E,d)$ is a complete separable metric space with distance~$d$, for any $p \geq 1$ we denote by $\mathcal{P}_p(E)$ the space of probability measures with finite $p$-th moments. For $\rho,\tilde{\rho}\in\mathcal{P}_p(E)$, the $p$-Wasserstein distance $W_p(\rho,\tilde{\rho})$ is defined by
\begin{equation*}
  W_p(\rho,\tilde{\rho}):= \inf_{\pi \in \Pi(\rho,\tilde{\rho})} \Big[ \int_{E \times E} d(x,y)^p \,\pi(\d x,\d y)\Big]^{1/p},
\end{equation*}
where $\Pi(\rho,\tilde{\rho})$ denotes the set of probability measures on $E\times E$ with marginals $\rho$ and $\tilde{\rho}$, see \cite[Chapter~5]{Carmona2018} for more details. The space $\R^d$ is equipped with the Euclidean norm~$| \cdot |$ and we set $\Delta_T:=\lbrace (s,t)\in [0,T]\times [0,T]\colon \, 0\leq s\leq t\leq T \rbrace$. For a measure space $\mathcal{X}$, a Banach spaces $\mathcal{Y}$ and $p\geq 1$, we use the notations $L^p(\mathcal{X};\mathcal{Y})$ for the space of all $\mathcal{Y}$-valued, measurable, $p$-integrable functions on $\mathcal{X}$, and we write $L^p(\mathcal{X}):=L^p(\mathcal{X};\R)$.

\medskip

Let $b\colon[0,T]\times\R^d\times\mathcal{P}_{\eta}(\R^d)\to\R^d$ and $\sigma\colon[0,T]\times\R^d\times\mathcal{P}_{\eta}(\R^{d})\to\R^{d\times m}$ for some $\eta \in \bN$ and the (Volterra) kernels $K_b, K_\sigma\colon \Delta_T\to \R$ be measurable functions. We consider the $d$-dimensional distribution-dependent stochastic Volterra equation (DDSVE)
\begin{equation}\label{eq:MFSVE}
  X_t = X_0 +\int_0^t K_{b}(s,t) b(s,X_s,\mathcal{L}(X_s))\dd s+\int_0^t K_{\sigma}(s,t)\sigma(s,X_s,\mathcal{L}(X_s))\dd B_s,\quad t\in [0,T],
\end{equation}
with initial distribution $\L(X_0)=\mu_0 \in \mathcal{P}_{p}(\R^d)$, that is, $X_0$ is a $d$-dimensional, $\mathcal{F}_0$-measurable random variable with distribution $\mu_0$ and with finite $p$-th moments, $p \geq \eta$, which is independent of $B$. The integral $\int_0^t K_{\sigma}(s,t)\sigma(s,X_s,\mathcal{L}(X_s))\dd B_s$ is defined as a stochastic It{\^o} integral.

\medskip

Let us recall the concepts of well-posedness, strong and weak solutions and pathwise uniqueness. An $(\mathcal{F}_t)$-progressively measurable stochastic process $(X_t)_{t\in [0,T]}$ in $L^p(\Omega\times [0,T];\R^d)$, on the given probability space $(\Omega,\mathcal{F},(\mathcal{F}_t)_{t\in[0,T]},\mathbb{P})$, is called a \textit{(strong) $L^p$-solution} of the DDSVE~\eqref{eq:MFSVE} if
\begin{equation*}
  \int_0^t (|K_b(s,t)b(s,X_s,\mathcal{L}(X_s))|+|K_\sigma(s,t)\sigma(s,X_s,\mathcal{L}(X_s))|^2 )\dd s<\infty \quad \text{for all }t\in[0,T],
\end{equation*}
and the integral equation~\eqref{eq:MFSVE} holds $\mathbb{P}$-almost surely. We say \textit{pathwise uniqueness in} $L^p$ holds for the DDSVE~\eqref{eq:MFSVE} if $\mathbb{P}(X_t=\tilde{X}_t, \,\forall t\in [0,T])=1$ for any two $L^p$-solutions $(X_t)_{t\in[0,T]}$ and $(\tilde{X}_t)_{t\in[0,T]}$ of \eqref{eq:MFSVE} defined on the same probability space $(\Omega,\mathcal{F},(\mathcal{F}_t)_{t\in[0,T]},\mathbb{P})$. We say that the DDSVE~\eqref{eq:MFSVE} is \textit{well-posed in $L^p$} (or that there exists a \textit{unique $L^p$-solution}) for $p\geq 1$ if there exists a strong $L^p$-solution to \eqref{eq:MFSVE} and pathwise uniqueness in $L^p$ holds. A \textit{weak solution} to \eqref{eq:MFSVE} is a triple $(X,B),(\Omega,\mathcal{F},\mathbb{P}),(\mathcal{F}_t)_{t\in [0,T]}$ such that
\begin{enumerate}
  \item[(i)] $(\Omega,\mathcal{F},\mathbb{P})$ is a probability space, $(\mathcal{F}_t)_{t\in [0,T]}$ is a filtration of sub-$\sigma$-algebras of $\mathcal{F}$ satisfying the usual conditions,
  \item[(ii)] $X=(X_t)_{t\in [0,T]}\in L^1(\Omega~\times~ [0,T]; \mathbb{R}^d)$ is an $(\mathcal{F}_t)$-progressively measurable stochastic process, $B=(B_t)_{t\in [0,T]}$ is an $m$-dimensional Brownian motion w.r.t. $(\mathcal{F}_t)_{t\in [0,T]}$,
  \item[(iii)] $\int_0^t\big(|K_b(s,t)b(s,X_s,\mathcal{L}(X_s))| + |K_\sigma(s,t)\sigma(s,X_s,\mathcal{L}(X_s))|^2  \big)\dd s<\infty$ $\P$-a.s. for any $t\in[0,T]$,
  \item[(iv)] \eqref{eq:MFSVE} holds for $(X,B)$ on $(\Omega,\mathcal{F},\mathbb{P})$, $\P$-a.s., and
  \item[(v)] $X_0$ is $\mu_0$-distributed.
\end{enumerate}
If a weak solution $(X_t)_{t \in [0,T]}$ to \eqref{eq:MFSVE} is in $L^p(\Omega \times [0,T];\bR^d)$, we may call it a \textit{weak $L^p$-solution}. We say that \textit{uniqueness in law} holds for the DDSVE~\eqref{eq:MFSVE} if any two weak solutions with the same initial distribution have the same law.

\begin{remark}
  In the following, we impose assumptions on the coefficients and kernels that ensure the existence of either a strong or a weak solution to \eqref{eq:MFSVE}, according to the results of \cite{Promel_MF} and \cite{Bergerhausen2026}, respectively. In \cite{Promel_MF}   strong well-posedness of \eqref{eq:MFSVE} is established in a multi-dimensional setting assuming suitable integrability and regularity on the kernels as well as Lipschitz continuity and a linear growth condition on the coefficients, and in a one-dimensional setting assuming sufficiently smooth kernels and H{\"o}lder continuous diffusion coefficients which are independent of the law of the solution. Replacing the Lipschitz continuity assumption by allowing the coefficients to be only continuous in space and law uniformly in time and generalizing the kernel assumption, the existence of a weak solution is obtained in \cite{Bergerhausen2026}. For further well-posedness results for DDSVEs, we refer to \cite{Shi2013,Liu2023,Jie2024,Kalinin2026}.
\end{remark}

Under suitable assumptions on the coefficients and kernels, the existence of weak solutions to the DDSVE~\eqref{eq:MFSVE} can be equivalently formulated in terms of solutions to an associated local martingale problem, see \cite[Proposition~2.8]{Bergerhausen2026}. To formulate an martingale problem for DDSVEs analogous to the classical (local) martingale problem associated to a stochastic differential equation (cf. e.g. \cite[Chapter~6]{Stroock1979} and \cite[Chapter~5]{Karatzas1991}), we assume that
\begin{equation*}
  Z\colon C([0,T];\R^d) \to \R^d, \quad \text{with}
  \quad Z_t(\omega):= \omega(t)
  \quad \text{for }\omega\in C([0,T];\R^d),
\end{equation*}
is the canonical process on the path space $C([0,T];\R^d)$, where $C([0,T];\R^d)$ denotes the space of all continuous functions $\omega\colon [0,T]\to \R^d$ equipped with the supremum norm $\|\cdot\|_{\infty}$ and $\mathcal{B}(C([0,T];\R^d))$ the Borel $\sigma$-algebra on $C([0,T];\R^d)$. Moreover, let $ C^2(\R^d)$ be the space of twice continuously differentiable functions $f\colon\R^d\to \R$ and $C_0^2(\R^d)$ be the space of all $f\in C^2(\R^d)$ with compact support.

\begin{definition}\label{def:martproblem_new}
  A probability measure $P$ on $(C([0,T];\R^d),\mathcal{B}(C([0,T];\R^d)))$ is called \textup{solution to the distribution-dependent Volterra local martingale problem} given $(\mu_0,b,\sigma,K_b,K_\sigma)$ if
  \begin{enumerate}
    \item[(i)] $(Z_t)_{t\in[0,T]}$ is a $d$-dimensional semimartingale and $Z_0 \sim \mu_0$, 
    \item[(ii)]the process $(\mathcal{M}^f_t)_{t\in[0,T]}$, given by
    \begin{equation}\label{def:M_f_new}
      \mathcal{M}^f_t:= f(Z_t)-\int_0^t \mathcal{A}^f(s,X_s,\L(X_s),Z_s)\dd s,\quad \mathcal{F}_t, \quad t\in [0,T],
    \end{equation}
    is a local martingale for every $f\in C_0^2(\R^d)$, where 
	\begin{equation}\label{eq:operator A}
      \mathcal{A}^{f}(t,x,\rho,z):=\sum_{i=1}^d b_i(t,x,\rho)\frac{\partial f(z)}{\partial z_i}+\frac{1}{2}\sum_{i,j=1}^d (\sigma\sigma^\top)_{ij}(t,x,\rho) \frac{\partial^2 f(z)}{\partial z_i \partial z_j},
	\end{equation}
	and
    \begin{equation}\label{eq:X_MP_new}
      X_t := Z_0 +\int_0^t K_b(s,t)\dd A_s + \int_0^tK_\sigma(s,t)\dd M_s, \quad t\in[0,T],
      \quad P\text{-a.s},
    \end{equation}
	$(A_t)_{t\in[0,T]}$ is predictable of bounded variation and $(M_t)_{t\in[0,T]}$ is a local martingale with $M_0=0$, such that $Z=A+M$. Here $\mathcal{F}_t = \mathcal{G}_{t+}$, and $(\mathcal{G}_t)_{t \in [0,T]}$ is the augmentation under $P$ of the canonical filtration $(\mathcal{B}_t)_{t \in [0,T]}$ with $\mathcal{B}_t := \sigma(Z_s\, :\, s\in [0,t])$.
  \end{enumerate}
\end{definition}

\begin{remark}
  The formulation of the distribution-dependent Volterra local martingale problem, presented in Definition~\ref{def:martproblem_new}, was introduced in \cite[Definition~2.5]{Bergerhausen2026}. In \cite[Proposition~2.8]{Bergerhausen2026}, it is shown that the existence of a weak solution of the DDSVEs~\eqref{eq:MFSVE} is indeed equivalent to the existence of a weak solution under fairly general assumptions. Furthermore, in the case of $K = 1$, we have $X = Z$ and the distribution-dependent Volterra local martingale problem reduces to the standard martingale problem for (McKean--Vlasov) stochastic differential equations, see e.g. \cite{Stroock1979,Funaki1984}.
\end{remark}

\subsection{Assumptions and remarks on the Volterra kernels}\label{subsec: kernel assumptions}

Throughout the paper, we impose the following assumption on the kernels. Henceforth, we fix the constants $\gamma$ and $\varepsilon$ as defined therein.

\begin{assumption}\label{ass:kernel}
  We say that a pair $(K_b,K_\sigma)$ satisfies this assumption given constants $(\gamma,\epsilon, L)$, $\gamma\in (0,\frac{1}{2}]$, $L>0$ and $\epsilon>0$, if $K_b, K_\sigma \colon \Delta_T\to \mathbb{R}$, are measurable functions fulfilling
    \begin{align*}
      &\int_0^t |K_{b}(s,t')-K_{b}(s,t)|^{1+\epsilon}\dd s + \int_t^{t'} |K_{b}(s,t')|^{1+\epsilon}\dd s \leq L|t'-t|^{\gamma(1+\epsilon)},\\
      &\int_0^t |K_{\sigma}(s,t')-K_{\sigma}(s,t)|^{2+\epsilon}\dd s  + \int_t^{t'} |K_{\sigma}(s,t')|^{2+\epsilon}\dd s  \leq L|t'-t|^{\gamma(2+\epsilon)},
    \end{align*}
    for all $(t,t^\prime)\in \Delta_T$.
\end{assumption}

Let us briefly discuss some auxiliary consequences of the preceding assumption, as well as several examples.

\begin{remark}\label{rem:asskernel_to_assshort}
  Note that kernels $(K_b,K_\sigma)$ satisfying Assumption~\ref{ass:kernel} also satisfy \cite[Assumption~2.2]{Bergerhausen2026} since
  \begin{equation*}
    \int_0^{t} |K_b(s,t)| \dd s \leq t^{\frac{\epsilon}{1+\epsilon}} \Big( \int_0^{t} |K_b(s,t)|^{1+\epsilon} \dd s\Big)^{\frac{1}{1+\epsilon}} \leq L^{\frac{1}{1+\epsilon}}\, t^{\gamma+\frac{\epsilon}{1+\epsilon}}
  \end{equation*}
  and
  \begin{equation*}
    \int_0^{t} |K_\sigma(s,t)|^2 \dd s \leq t^{\frac{\epsilon}{2+\epsilon}} \Big( \int_0^{t} |K_{\sigma}(s,t)|^{2+\epsilon} \dd s\Big)^{\frac{2}{2+\epsilon}} \leq L^{\frac{2}{2+\epsilon}}\, t^{2\gamma+\frac{\epsilon}{2+\epsilon}},
  \end{equation*}
  that is, $K_b(\cdot,t) \in L^1([0,t])$ and $K_\sigma(\cdot,t) \in L^2([0,t])$, for every $ t \in [0,T]$. Moreover, let us remark, because of
  \begin{align*}
    \int_0^T \int_0^{t} |K_b(s,t)| \dd s \dd t \leq \tilde{L} \int_0^T t^{\gamma+\frac{\epsilon}{1+\epsilon}} \dd t \leq \tilde{L}\frac{T^{\gamma+\frac{\epsilon}{1+\epsilon}+1}}{\gamma+\frac{\epsilon}{1+\epsilon}+1}
  \end{align*}
  and
  \begin{align*}
    \int_0^T \int_0^{t} |K_\sigma(s,t)|^2 \dd s \dd t \leq \tilde{L} \int_0^T t^{2\gamma+\frac{\epsilon}{2+\epsilon}} \dd t \leq \tilde{L}\frac{T^{2\gamma+\frac{\epsilon}{2+\epsilon}+1}}{2\gamma+\frac{\epsilon}{2+\epsilon}+1},
  \end{align*}
  we have $K_b \in L(\Delta_T)$ and $K_\sigma \in L^2(\Delta_T)$.
\end{remark}

\begin{remark}\label{rmk:BDG inequality}
  In general, under general regularity assumptions on the kernels, as they are imposed in Assumption~\ref{ass:kernel}, the stochastic process
  \begin{align*}
    \Big( \int_0^t K_\sigma(s,t)\sigma(s,X_s,\mathcal{L}(X_s)) \dd B_s\Big)_{t \in [0,T]}
  \end{align*}
  is not a local martingale. However, fixing the second argument in the kernel, the stochastic process
  \begin{align*}
    \Big( \int_0^t K_\sigma(s,\tilde{t})\sigma(s,X_s,\mathcal{L}(X_s)) \dd B_s\Big)_{t \in [0,\tilde{t}]}
  \end{align*}
  is a local martingale for every $\tilde{t} \in (0,T]$. Hence, we can apply the classical Burkholder--Davis--Gundy inequality to this martingale to obtain
  \begin{align*}
    \bE \Big[ \sup_{t \in [0,\tilde{t}]} \Big| \int_0^t K_\sigma(s,\tilde{t}) \sigma(s,X_s,\mathcal{L}(X_s))\dd B_s\Big|^p \Big] \leq C_p \bE\Big[ \Big( \int_0^{\tilde{t}} (K_\sigma(s,\tilde{t}) \sigma (s,X_s,\mathcal{L}(X_s)))^2 \dd s \Big)^{p/2} \Big],
  \end{align*}
  $\tilde{t} \in [0,T]$, for some positive $C_p$ and, in particular,
  \begin{align}\label{eq:BDG}
    \bE \Big[ \Big| \int_0^{\tilde{t}} K_\sigma(s,\tilde{t}) \sigma(s,X_s,\mathcal{L}(X_s))\dd B_s \Big|^p \Big] \leq C_p \bE\Big[ \Big( \int_0^{\tilde{t}} (K_\sigma(s,\tilde{t}) \sigma (s,X_s,\mathcal{L}(X_s)))^2 \dd s \Big)^{p/2} \Big]
  \end{align}
  holds for every $\tilde{t} \in [0,T]$. In the following, we often refer to the Burkholder--Davis--Gundy inequality meaning that we use \eqref{eq:BDG}.
\end{remark}

\begin{remark}
  Assumption~\ref{ass:kernel} is satisfied, e.g., by the following type of diffusion kernels:
  \begin{enumerate}
    \item[(i)] $K_\sigma(s,t):=\tilde{K}(t-s)$ for a Lipschitz continuous function $\tilde{K}\colon [0,T]\to \R$,
    \item[(ii)] $K_\sigma(s,t):=C(t-s)^{-\alpha}$ for $\alpha\in(0,\frac{1}{2})$ with $\varepsilon \in (0,\frac{1}{\alpha}-2)$ and $\gamma \in (0,\frac{1}{2+\varepsilon}-\alpha]$,
    \item[(iii)] weakly differentiable kernels such that $\partial_1 K_\sigma(s,t)\leq C(t-s)^{-\alpha}$ for $\alpha\in(0,\frac{1}{2})$, and
    \item[(iv)] a finite combination of exponential kernel $K_\sigma(s,t) = \sum_{i=1}^k c_i \exp(- \theta_i (t-s))$ with $k \in \bN$, $c_i >0$ and $\theta_i \in \bR^+$ for $i = 1, \ldots, k$ with $\gamma \in (0, \frac{1}{2+ \varepsilon}]$,
    \item[(v)] the gamma kernel $K_\sigma(s,t) = \frac{1}{\Gamma(\alpha)} \exp(- \beta (t-s)) (t-s)^{\alpha -1}$ with $\beta >0$ and exponents $\alpha \in (\frac{1}{2},1)$ with $\varepsilon< \frac{2 \alpha - 1}{1-\alpha}$ and $\gamma = \alpha - \frac{1+\varepsilon}{2+\varepsilon}$,
    \item[(vi)] kernels fulfilling \cite[Assumption~2.1]{Promel2023}.
  \end{enumerate}
  In all cases one needs an initial condition with $p> \max\{\frac{2 \eta +1}{\gamma}, 2+\frac{4}{\varepsilon}\}$ finite moments.
\end{remark}

\section{Stability results for strong solutions to DDSVEs}\label{subsec:strong_results}

In this section, we investigate the stability of strong solutions to distribution-dependent stochastic Volterra equations with respect to their coefficients. We begin in Subsection~\ref{subsec:Lipschitz coefficients} by deriving a quantitative stability estimate under Lipschitz continuity and linear growth assumptions on the coefficients. We then establish a general convergence result in Subsection~\ref{subsec: continuity assumptions}, where the Lipschitz continuity assumption is replaced by local continuity of the limiting coefficients, uniformly in time, together with a uniform linear growth condition on the approximating coefficients.

\subsection{Quantitative stability estimate under Lipschitz assumptions}\label{subsec:Lipschitz coefficients}

In this subsection, we shall prove a quantitative stability estimate under Lipschitz and linear growth assumptions on the coefficients. To that end, as comparison to the DDSVE~\eqref{eq:MFSVE}, we consider the DDSVE
\begin{equation}\label{eq:MFSVE2}
  \tilde{X}_t = \tilde{X}_0+\int_0^t \tilde{K}_{b}(s,t)\tilde{b}(s,\tilde{X}_s,\mathcal{L}(\tilde{X}_s))\dd s+\int_0^t \tilde{K}_{\sigma}(s,t)\tilde{\sigma}(s,\tilde{X}_s,\mathcal{L}(\tilde{X}_s))\dd B_s,\quad t\in [0,T],
\end{equation}
where $\tilde{X}_0$ is a $d$-dimensional, $\mathcal{F}_0$-measurable random variable, and where the coefficients $\tilde{b}\colon [0,T]\times\R^d \times\mathcal{P}_{\eta}(\R^d)\to\R^d $ and $\tilde{\sigma}\colon [0,T]\times\R^d\times\mathcal{P}_{\eta}(\R^d)\to\R^{d\times m}$, and the kernels $\tilde{K}_b, \tilde{K}_\sigma \colon \Delta_T \to \R$ are measurable functions.

\medskip

We need to impose the following Lipschitz and linear growth condition on the coefficients:

\begin{assumption}\label{ass:coefficients1}
  Let $b, \tilde{b}\colon [0,T]\times\R^d\times\mathcal{P}_{\eta}(\R^d)\to\R^d $ and $\sigma,\tilde{\sigma}\colon [0,T]\times\R^d\times\mathcal{P}_{\eta}(\R^d)\to\R^{d\times m} $ be measurable functions such that:
  \begin{enumerate}
    \item[(i)] for any bounded set $\mathcal{K}\subset \mathcal{P}_{\delta}(\R^d)$, there is a constant $C_{\mathcal{K}}>0$, such that the linear growth condition
    \begin{equation*}
      |b(t,x,\rho)|+|\tilde{b}(t,x,\rho)|+|\sigma(t,x,\rho)|+|\tilde{\sigma}(t,x,\rho)|\leq C_{\mathcal{K}} (1+|x|)
    \end{equation*}
    holds for all $\rho \in \mathcal{K}$, $t\in[0,T]$ and $x\in\R^d$;
    \item[(ii)] $b$, $\tilde{b}$, $\sigma$ and $\tilde{\sigma}$ are Lipschitz continuous in $x$ and in $\rho$ w.r.t. the $\eta$-Wasserstein distance, uniformly in time, i.e. there is a constant $C_{b,\sigma}>0$ such that
    \begin{align*}
      |b(t,x,\rho)-b(t,y,\nu)|+|\sigma(t,x,\rho)-\sigma(t,y,\nu)|
      \leq C_{b,\sigma}\big(|x-y|+W_{\eta}(\rho,\nu)\big),\\
      |\tilde{b}(t,x,\rho)-\tilde{b}(t,y,\nu)|+|\tilde{\sigma}(t,x,\rho)-\tilde{\sigma}(t,y,\nu)|
      \leq C_{b,\sigma}\big(|x-y|+W_{\eta}(\rho,\nu)\big),
    \end{align*}
    holds for all $t\in [0,T]$, $x,y\in \R^d$, and $\rho,\nu\in \mathcal{P}_{\eta}(\R^d)$.
  \end{enumerate}
  Furthermore for $\delta := 2 + \frac{4}{\epsilon}$ and $p > \max\{\frac{1}{\gamma}, \delta\}$, $p \geq \eta$ we let $\bE[|X_0|^p] < \infty$ and $\bE[|\tilde{X}_0|^p] < \infty$.
\end{assumption}

For $q:= \frac{p}{p-1}$ and $\tilde{q}:= \frac{p}{p-2}$, we note that
\begin{equation*}
  \frac{1}{p}+ \frac{1}{q} = 1 \qquad \text{ and } \qquad \frac{2}{p} + \frac{1}{\tilde{q}} = 1
\end{equation*}
and furthermore
\begin{equation*}
  \|K_b\|_q + \|\tilde{K}_b\|_q < \infty \qquad \text{ and } \qquad \|K_\sigma\|_{2 \tilde{q}} + \|\tilde{K}_\sigma\|_{2 \tilde{q}} < \infty,
\end{equation*}
where
\begin{equation*}
  \|h\|_p := \Big( \sup_{t\in[0,T]} \int_0^t |h(s,t)|^p \dd s \Big)^{\frac{1}{p}}.
\end{equation*}

The next theorem presents the quantitative stability estimate for DDSVEs under Lipschitz and linear growth assumptions on the coefficients.

\begin{theorem}\label{theorem:stability_MFSVE}
  Suppose  $(K_b,K_{\sigma})$ and $(\tilde{K}_b,\tilde{K}_{\sigma})$ satisfy Assumption~\ref{ass:kernel}, Assumption~\ref{ass:coefficients1} is satisfied and assume $X_0, \tilde{X}_0 \in L^p(\Omega)$. Let $(X_t)_{t\in[0,T]}$ and $(\tilde{X}_t)_{t\in[0,T]}$ be the strong solution to the DDSVE \eqref{eq:MFSVE} and \eqref{eq:MFSVE2}, respectively. Then, there is some constant $C>0$, depending on $b$, $\sigma$, $\tilde{b}$, $\tilde{\sigma}$, ${K}_b$, ${K}_\sigma$, $\tilde{K}_b$, $\tilde{K}_\sigma$ and $p$, such that
  \begin{equation}
    \sup_{t\in [0,T]}\E[|X_t-\tilde{X}_t|^p]\leq C \Big( \bE[|X_0-\tilde{X}_0|^p]+\|b-\tilde{b}\|_{\infty}^p+\|\sigma-\tilde{\sigma}\|_{\infty}^p+\|K_b-\tilde{K}_b\|_{q}^p + \|K_\sigma -\tilde{K}_\sigma\|_{2\tilde{q}}^p \Big).
  \end{equation}
\end{theorem}

\begin{proof}
  Let $t\in [0,T]$ and $C>0$ be a generic constant, which may change from line to line. We get that
  \begin{align*}
    & \E\big[ |X_t-\tilde{X}_t|^p \big]\\
    &= \E\bigg[ \bigg| X_0 - \tilde{X}_0 + \int_0^t K_b(s,t)b(s,X_s,\mathcal{L}(X_s))\dd s - \int_0^t \tilde{K}_b(s,t)\tilde{b}(s,\tilde{X}_s,\mathcal{L}(\tilde{X}_s))\dd s\\
    &\quad + \int_0^t K_\sigma(s,t)\sigma(s,X_s,\mathcal{L}(X_s))\dd B_s - \int_0^t \tilde{K}_\sigma(s,t)\tilde{\sigma}(s,\tilde{X}_s,\mathcal{L}(\tilde{X}_s))\dd B_s \bigg|^p\bigg]\\
    &\leq C \bigg( \bE[|X_0 - \tilde{X}_0|^p] + \E\Big[\Big|\int_0^t \big(K_b(s,t)-\tilde{K}_b(s,t)\big)b(s,X_s,\mathcal{L}(X_s))\dd s\Big|^p\Big] \\
    &\qquad+ \E\Big[\Big| \int_0^t \tilde{K}_b(s,t)\big(b(s,X_s,\mathcal{L}(X_s))-\tilde{b}(s,\tilde{X}_s,\mathcal{L}(\tilde{X}_s))\big)\dd s\Big|^p\Big]\\
    &\qquad + \E\Big[\Big|\int_0^t \big(K_\sigma(s,t)-\tilde{K}_\sigma(s,t)\big)\sigma(s,X_s,\mathcal{L}(X_s))\dd B_s\Big|^p\Big]\\
    &\qquad + \E\Big[\Big| \int_0^t \tilde{K}_\sigma(s,t)\big(\sigma(s,X_s,\mathcal{L}(X_s))-\tilde{\sigma}(s,\tilde{X}_s,\mathcal{L}(\tilde{X}_s))\big)\dd B_s\Big|^p\Big] \bigg)\\
    &\leq C \bigg( \bE[|X_0 - \tilde{X}_0|^p] + \Big( \int_0^t \big|K_b(s,t)-\tilde{K}_b(s,t)\big|^q\dd s\Big)^{\frac{p}{q}} \E \Big[ \int_0^t\big|b(s,X_s,\mathcal{L}(X_s))\big|^p\dd s\Big] \\\
    &\qquad+\Big(\int_0^t \big|\tilde{K}_b(s,t)\big|^q\dd s\Big)^{\frac{p}{q}} \E\Big[ \int_0^t \big|b(s,X_s,\mathcal{L}(X_s))-\tilde{b}(s,\tilde{X}_s,\mathcal{L}(\tilde{X}_s))\big|^p\dd s\Big]\\
    &\qquad + \|K_\sigma - \tilde{K}_\sigma\|^p_{2 \tilde{q}} \int_0^t \E[|\sigma(s,X_s,\mathcal{L}(X_s))|^p]\dd s\\
    &\qquad + C \int_0^t \E[|\sigma(s,X_s,\mathcal{L}(X_s))-\tilde{\sigma}(s,\tilde{X}_s,\mathcal{L}(\tilde{X}_s))|^p] \dd s \bigg)\\
    &\leq C \bigg( \bE[|X_0 - \tilde{X}_0|^p] + \Big( \int_0^t \big|K_b(s,t)-\tilde{K}_b(s,t)\big|^q\dd s\Big)^{\frac{p}{q}} \E \Big[ \int_0^t\big|b(s,X_s,\mathcal{L}(X_s))\big|^p\dd s\Big] \\\
    &\qquad+\Big(\int_0^t \big|\tilde{K}_b(s,t)\big|^q\dd s\Big)^{\frac{p}{q}} \bigg(\E\Big[ \int_0^t \big|b(s,X_s,\mathcal{L}(X_s))-b(s,\tilde{X}_s,\mathcal{L}(\tilde{X}_s))\big|^p\dd s\Big]\\
    &\qquad \qquad +\E\Big[ \int_0^t \big|b(s,\tilde{X}_s,\mathcal{L}(\tilde{X}_s))-\tilde{b}(s,\tilde{X}_s,\mathcal{L}(\tilde{X}_s))\big|^p\dd s\Big]\bigg)\\
    &\qquad + \|K_\sigma - \tilde{K}_\sigma\|^p_{2 \tilde{q}} \int_0^T \E[|\sigma(s,X_s,\mathcal{L}(X_s))|^p]\dd s\\
    &\qquad + C \bigg(\int_0^t \E[|\sigma(s,X_s,\mathcal{L}(X_s))-\sigma(s,\tilde{X}_s,\mathcal{L}(\tilde{X}_s))|^p] \dd s\\
    &\qquad\qquad + \int_0^t \E[|\sigma(s,\tilde{X}_s,\mathcal{L}(\tilde{X}_s))-\tilde{\sigma}(s,\tilde{X}_s,\mathcal{L}(\tilde{X}_s))|^p] \dd s \bigg) \bigg).
  \end{align*}
    
  Recall that $p \geq \eta$ and note that
  \begin{equation}\label{eq:Wasserstein_bound}
    W_\eta(\delta_0,\L(X_s))^p \leq \bE[|X_s|^\eta]^{\frac{p}{\eta}} \leq \bE[|X_s|^p].
  \end{equation}
  By \cite[Lemma~4.1]{Bergerhausen2026}, one obtains the boundedness of the first $p$ moments of distribution-dependent Volterra processes. Hence, \eqref{eq:Wasserstein_bound} is uniformly bounded. Using the regularity assumptions on $b$ and $\sigma$ (Assumption~\ref{ass:coefficients1}) and again the boundedness of the first $p$ moments of distribution-dependent Volterra processes, we get
  \begin{align*}
    & \E\big[|X_t-\tilde{X}_t|^p \big]\\
    &\leq C \bigg( \bE[|X_0 - \tilde{X}_0|^p] + \|K_b-\tilde{K}_b\|_q^p + \|K_\sigma - \tilde{K}_\sigma\|_{2 \tilde{q}}^p\\
    &\qquad + \int_0^t \E\big[\big|b(s,X_s,\mathcal{L}(X_s))-b(s,\tilde{X}_s,\mathcal{L}(\tilde{X}_s))\big|^p\big] + \|b-\tilde{b}\|_\infty^p \dd s\notag\\
    &\qquad + \int_0^t\E\Big[\big|\sigma(s,X_s,\mathcal{L}(X_s))-\tilde{\sigma}(s,\tilde{X}_s,\mathcal{L}(\tilde{X}_s))\big|^p\Big]+ \|\sigma-\tilde{\sigma}\|_\infty^p\dd s \bigg).
  \end{align*}
  Note that by the Lipschitz assumption and the properties of the Wasserstein distance
  \begin{align*}
    \E\big[\big|b(s,X_s,\mathcal{L}(X_s))-b(s,\tilde{X}_s,\mathcal{L}(\tilde{X}_s))\big|^p\big] & \leq \E[C(|X_s-\tilde{X}_s| + W_\eta(\mathcal{L}(X_s),\mathcal{L}(\tilde{X}_s)))^p]\\
    & \leq C(\E[|X_s-\tilde{X}_s|^p] + \E[|X_s - \tilde{X}_s]^\eta]^\frac{p}{\eta})\\
    & \leq C \E[|X_s - \tilde{X}_s|^p].
  \end{align*}
  Hence, we obtain
  \begin{align*}
    \E\big[ |X_t-\tilde{X}_t|^p \big]
    \leq C &\bigg( \bE[|X_0 - \tilde{X}_0|^p] +\|b-\tilde{b}\|_{\infty}^p+\|\sigma-\tilde{\sigma}\|_{\infty}^p+\|K_b-\tilde{K}_b\|_{q}^p + \|K_\sigma -\tilde{K}_\sigma\|_{2\tilde{q}}^p\bigg)\\
    &\quad +C \int_0^t \E\big[\big|X_s-\tilde{X}_s\big|^p\big]\dd s.
  \end{align*}
  Applying Gr{\"o}nwall's lemma, one gets
  \begin{align*}
    \E\big[|X_t-\tilde{X}_t|^p \big]
    \leq C &\bigg( \bE[|X_0 - \tilde{X}_0|^p] + \|b-\tilde{b}\|_{\infty}^p+\|\sigma-\tilde{\sigma}\|_{\infty}^p+\|K_b-\tilde{K}_b\|_{q}^p + \|K_\sigma -\tilde{K}_\sigma\|_{2\tilde{q}}^p\bigg).
  \end{align*}
  Then, taking the supremum in $t\in [0,T]$ finishes the proof.
\end{proof}

\begin{remark}
  Theorem~\ref{theorem:stability_MFSVE} extends \cite[Proposition~2.1]{Bergerhausen2025_2} in two directions. First, it applies to general Volterra kernels $K_b,K_\sigma \colon \Delta_T\to\mathbb R$ satisfying Assumption~\ref{ass:kernel}, whereas \cite[Proposition~2.1]{Bergerhausen2025_2} is restricted to convolution kernels of the form $K(s,t)=k(t-s)$. Second, the coefficients considered here may depend on the marginal law of the solution, while \cite[Proposition~2.1]{Bergerhausen2025_2} treats coefficients depending only on time and state.
\end{remark}

\subsection{Stability result under continuity assumptions}\label{subsec: continuity assumptions}

In this subsection, we establish a general convergence result in which the Lipschitz continuity assumption from the previous subsection is relaxed to local continuity of the limiting coefficients, uniformly in time, together with a uniform linear growth condition on the approximating coefficients.

\medskip

For this purpose, we consider the distribution-dependent stochastic Volterra equation
\begin{align}\label{eq:MVSVE_2}
  \begin{split}
  X(t,x) = x &+\int_0^t K_{b}(s,t) b(s,X(s,x),\mathcal{L}(X(s,x)))\dd s\\
  &+\int_0^t K_{\sigma}(s,t)\sigma(s,X(s,x),\mathcal{L}(X(s,x)))\dd B_s, \qquad t\in [0,T],
  \end{split}
\end{align}
and analyse the behaviour of the solution $X_n(x,t)$, $n \in \bN$, of
\begin{align}\label{eq:MVSVE_2app}
  \begin{split}
  X_n(t,x) = x &+\int_0^t K_{b,n}(s,t) b_n(s,X_n(s,x),\mathcal{L}(X_n(s,x)))\dd s\\
  &+\int_0^t K_{\sigma,n}(s,t)\sigma_n(s,X_n(s,x),\mathcal{L}(X_n(s,x)))\dd B_s, \qquad t\in [0,T],
  \end{split}
\end{align}
as $(K_{b,n}, K_{\sigma,n}, b_n, \sigma_n)$, $n \in \bN$, approaches $(K_{b}, K_{\sigma}, b, \sigma)$ in a suitable way, assuming suitable measurability for all involved functions. More precisely, the assumption on the coefficients reads as follows.

\begin{assumption}\label{ass:volatility}
  There is an $\eta \geq 1$ such that the coefficients $b, b_n \colon [0,T] \times \mathbb{R}^d \times \mathcal{P}_\eta(\mathbb{R}^d)\to \bR^{d}$ and $\sigma, \sigma_n \colon [0,T] \times \mathbb{R}^d \times \mathcal{P}_\eta(\mathbb{R}^d)\to \bR^{d \times m}$ for all $n \in \bN$ are measurable functions such that
  \begin{itemize}
    \item there exists a constant $L>0$ such that, for all $(t,x,\mu) \in [0,T]\times \bR^d \times \mathcal{P}_{\eta}(\bR^d)$,
    \begin{align}\label{eq:lin_grow_bound}
      \sup_{n \in \bN} |\sigma_n(t,x,\mu)|+|b_n(t,x,\mu)| \leq L(1+|x|+W_\eta(\delta_0,\mu)),
    \end{align}
    \item for all compact sets $\mathcal K \subset \bR^d$, $\bar{\mathcal K}\subset \mathcal{P}_\eta(\bR^d)$ and every $\epsilon >0$ there exists a $\delta >0$ such that
  \begin{equation*}
    |\sigma(t,x,\mu)-\sigma(t,y,\nu)| + |b(t,x,\mu)-b(t,y,\nu)| \leq \epsilon
  \end{equation*}
  for all $t \in [0,T]$, $x,y \in \mathcal K$ satisfying $|x-y| \leq \delta$, and $\mu, \nu \in \bar{\mathcal K}$ with $W_\eta(\mu,\nu)\leq \delta$.
  \end{itemize}
\end{assumption}

The main result of this section is the following theorem.

\begin{theorem}\label{theorem:approximation}
  Suppose that
  \begin{itemize}
	\item Assumption \ref{ass:volatility} holds,
	\item all pairs of kernels $(K_{b,n}, K_{\sigma,n})$, $n \in \bN$, and $(K_{b}, K_{\sigma})$ satisfy Assumption~\ref{ass:kernel} using the same constants $\gamma, \epsilon, L$,
	\item \eqref{eq:MVSVE_2} and \eqref{eq:MVSVE_2app} have unique strong $L^p$-solutions for each $x \in \bR^d$ and $n \in \bN$, respectively, for $p\geq \eta$ and $p>2+\frac{4}{\epsilon}$,
    \item for any compact subset $K \subset \bR^d$ and $\mathcal K \subset \mathcal P_\eta(\bR^d)$
    \begin{equation*}
      \lim_{n \to \infty} \sup_{t \in [0,T]} \sup_{x \in K} \sup_{\mu \in \mathcal K} |\sigma_n(t,x,\mu)-\sigma(t,x,\mu)|+|b_n(t,x,\mu)-b(t,x,\mu)|=0.
    \end{equation*}
    Additionally
    \begin{equation*}
      \sup_{t \in [0,T]} \int_0^t |K_{b,n}(s,t)-K_b(s,t)|^{1+\epsilon} \dd s \to 0 \text{ as } n \to \infty
    \end{equation*}
    and
    \begin{equation*}
      \sup_{t \in [0,T]} \int_0^t |K_{\sigma,n}(s,t)-K_\sigma(s,t)|^{2+\epsilon} \dd s \to 0 \text{ as } n \to \infty.
    \end{equation*}
  \end{itemize}
  Then
  \begin{align*}
    \lim_{n \to \infty} \sup_{x \in K} \bE\big [\max_{t \in [0,T]} |X_n(t,x)-X(t,x)|^2\big] = 0
  \end{align*}
  for every compact $K \subset \bR^d$.
\end{theorem}

Note that the assumption on the convergence of $K_{b,n}$ to $K_b$ and of $K_{\sigma,n}$ to $K_\sigma$, as postulated in Theorem~\ref{theorem:approximation}, can be seen as convergence in the classical Lebesgue spaces with mixed integrality $L^{(1+\epsilon,\infty)}(\Delta_T)$ and $L^{(2+\epsilon,\infty)}(\Delta_T)$, respectively.

\medskip

As preparation for the proof of Theorem~\ref{theorem:approximation}, we first establish three auxiliary lemmas, which provide uniform moment bounds for the approximating sequence of solutions and a uniform H\"older continuity estimate.

\begin{lemma}\label{lemma:pathcontinuity}
  Under the assumptions of Theorem~\ref{theorem:approximation} $X_n$ is $\bP$-a.s. continuous for all  $n \in \bN$.
\end{lemma}

\begin{proof}
  H{\"o}lder continuity can be shown following the steps of the first part of the proof of \cite[Lemma~3.9]{Bergerhausen2026}.
\end{proof}

\begin{remark}
	One may replace the assumption on $p$ stated in Theorem~\ref{theorem:approximation} by requiring the solutions of \eqref{eq:MVSVE_2} and \eqref{eq:MVSVE_2app} to be continuous. Then one can skip Lemma~\ref{lemma:pathcontinuity} and directly move to Lemma~\ref{lemma:boundedness2} to obtain that the solutions are in $L^p$ for arbitrary large $p$.
\end{remark}

\begin{lemma}\label{lemma:boundedness2}
  Under the assumptions of Theorem~\ref{theorem:approximation} for every $q \in [1,\infty)$ and every compact set $K \subset \bR^d$ there is a $C_q >0$ such that $\bE[|X_n(t,x)|^q] \leq C_q$ for all $n \in \bN$, $t \in [0,T]$ and $x \in K$.
\end{lemma}

\begin{proof}
  Fix some $n \in \bN$ and $q\geq \eta$ sufficiently large such that $\tilde{q}:=\frac{q}{q-2}\leq 1+ \frac{\epsilon}{2}$, $q':= \frac{q}{q-1}\leq 1+ \epsilon$. We introduce the hitting times
  \begin{equation*}
    \tau_k := \inf_{t \in [0,T]}\{|X_n(t,x)| \geq k\} \wedge T,\quad\text{for } k \in \bN.
  \end{equation*}
  Note that $\tau_k \to T$ $\bP$-a.s. as $k \to \infty$, by the path continuity from Lemma~\ref{lemma:pathcontinuity}. By the D{\'e}but theorem (\cite[Chapter I, (4.15) Theorem]{Revuz1999}) the hitting times $\tau_k$, $k \in \bN$, are stopping times. Hence, we get that
  \begin{align*}
    &\bE[|X_n(t,x)|^q \mathbbm{1}_{\{t \leq \tau_k\}}]\\
    & \quad = \bE\Big[\Big|x + \int_0^t K_{b,n}(s,t) b_n \big(s,X_n(s,x),\L(X_n(s,x))\big) \dd s\\
    &  \quad \qquad + \int_0^t K_{\sigma,n}(s,t) \sigma_n \big(s,X_n(s,x),\L(X_n(s,x))\big) \dd B_s \Big|^q \mathbbm{1}_{\{t \leq \tau_k\}} \Big]\\
    & \quad  = \bE\Big[\Big|x \mathbbm{1}_{\{t \leq \tau_k\}} + \int_0^t K_{b,n}(s,t) b_n \big(s,X_n(s,x),\L(X_n(s,x))\big) \dd s \mathbbm{1}_{\{t \leq \tau_k\}}\\
    & \quad  \qquad + \int_0^t K_{\sigma,n}(s,t) \sigma_n \big(s,X_n(s,x),\L(X_n(s,x))\big) \dd B_s \mathbbm{1}_{\{t \leq \tau_k\}} \Big|^q\Big]\\
    & \quad  \leq C_q \Big(|x|^q + \bE\Big[\Big| \int_0^t K_{b,n}(s,t) b_n \big(s,X_n(s,x),\L(X_n(s,x))\big) \big) \mathbbm{1}_{\{s \leq \tau_k\}} \dd s\Big|^q\\
    &  \quad \qquad + \bE\Big[\Big|\int_0^t K_{\sigma,n}(s,t) \sigma_n \big(s,X_n(s,x),\L(X_n(s,x))\big) \mathbbm{1}_{\{s \leq \tau_k\}} \dd B_s\Big|^q\Big]\Big)\\
    &  \quad \leq C_q \Big(|x|^q + \bE\Big[\Big| \int_0^t K_{b,n}(s,t) b_n \big(s,X_n(s,x),\L(X_n(s,x))\big) \big) \mathbbm{1}_{\{s \leq \tau_k\}} \dd s\Big|^q\\
    &  \quad \qquad + \bE\Big[\Big(\int_0^t |K_{\sigma,n}(s,t) \sigma_n \big(s,X_n(s,x),\L(X_n(s,x))\big) \mathbbm{1}_{\{s \leq \tau_k\}}|^2 \dd s\Big)^{q/2}\Big]\Big)\\
    &  \quad \leq C_{q,T} \bigg(|x|^q \\
    & \quad  \qquad +\Big(\int_0^t |K_{b,n}(s,t)|^{q'} \dd s\Big)^\frac{q}{q'} \int_0^t \bE[| b_n \big(s,X_n(s,x),\L(X_n(s,x))\big)|^q \mathbbm{1}_{\{s \leq \tau_k\}}]\dd s\\
    &  \quad \qquad + \Big(\int_0^t |K_{\sigma,n}(s,t)|^{2 \tilde{q}} \dd s\Big)^\frac{q}{2 \tilde{q}} \int_0^t \bE[|\sigma_n \big(s,X_n(s,x),\L(X_n(s,x))\big)|^q \mathbbm{1}_{\{s \leq \tau_k\}}]\dd s\bigg),
  \end{align*}
  where we used the Burkholder--Davis--Gundy inequality (cf. Remark~\ref{rmk:BDG inequality}) in the penultimate inequality and the H{\"o}lder inequality in the last step.

  Recall that $q \geq \eta$ and note that
  \begin{equation*}
    W_\eta(\delta_0,\L(X^m(\kappa_m(s))))^q \leq \bE[|X^m(\kappa_m(s))|^\eta]^{\frac{q}{\eta}} \leq \bE[|X^m(\kappa_m(s))|^q].
  \end{equation*}
  Using the linear growth assumption in Assumption~\ref{ass:volatility}, the above fact and Assumption~\ref{ass:kernel}, we have
  \begin{align*}
    \bE[|X_n(t,x)|^q \mathbbm{1}_{\{t \leq \tau_k\}}] & \leq C_{q,T} \Big(|x|^q + C_{K,b,\sigma,L} \int_0^t 1+2 \bE[|X_n(s,x)|^q \mathbbm{1}_{\{s \leq \tau_k\}}] \dd s\Big)\\
    & \leq C_{q,T,b,\sigma,k,L} \Big(1+|x|^q + \int_0^t \bE[|X_n(s,x)|^q \mathbbm{1}_{\{s \leq \tau_k\}} \dd s \Big).
  \end{align*}
  Using that $t \mapsto \bE[|X_n(t,x)|^q \mathbbm{1}_{\{t \leq \tau_k\}}]$ is bounded, we can apply Gronwall's lemma (see e.g. \cite[Lemma~26.9]{Klenke2020}) to get
  \begin{equation*}
    \bE[|X_n(t,x)|^q \mathbbm{1}_{\{t \leq \tau_k\}}] \leq  C_{p,T,b,\sigma,k,L} (1+|x|^q).
  \end{equation*}
  Sending $k \to \infty$ and noting that $K$ is compact, we have the assertion. By the orderedness of $L^p$-spaces the general statement follows.
\end{proof}

For the following, we define the sequences $(A^n)_{n \in \bN}$ and $(M^n)_{n \in \bN}$ by
\begin{equation*}
  A_n(t,x) := \int_0^t b_n(s,X_n(s,x), \L(X_n(s,x)))\dd s , \quad t \in [0,T], \, n \in \bN,
\end{equation*}
and
\begin{equation*}
  M_n(t,x) := \int_0^t \sigma_n(s,X_n(s,x),\L(X_n(s,x))) \dd B_s, \quad t \in [0,T], \, n \in \bN.
\end{equation*}

\begin{lemma}\label{lem:tightness}
  Under the assumptions of Theorem~\ref{theorem:approximation} for all compact $K \subset \bR^d$ there is an $C > 0$ such that
  \begin{align*}
    \bE[|X_n(t,x)-X_n(t',x)|^p] & \leq C|t'-t|^{\gamma p},\\
    \bE[|M_n(t,x)-M_n(t',x)|^p] & \leq C|t'-t|^{\gamma p},\\
    \bE[|A_n(t,x)-A_n(t',x)|^p] & \leq C|t'-t|^{\gamma p},
  \end{align*}
  for all $n \in \bN$, $x \in K$, $(t,t')\in \Delta_T$, $p \geq 1$.
\end{lemma}

\begin{proof}
  Since the linear growth bounds in \eqref{eq:lin_grow_bound} are assumed to hold uniformly and the kernels are assumed to satisfy Assumption~\ref{ass:kernel} with the same constants, using the uniform bound in Lemma~\ref{lemma:boundedness2} the lemma can be proven following the exact same steps as in the proof of \cite[Lemma~3.9]{Bergerhausen2026}.
\end{proof}

With these results at hand we can turn to the proof of this section's main result:

\begin{proof}[Proof of Theorem~\ref{theorem:approximation}]
  Suppose, for the sake of contradiction, that the theorem is wrong. Then, there are $\delta > 0$, an increasing sequence $(n_k)_{k \in \bN} \subset \bN$, a compact set $K \subset \bR^d$ and a sequence $(x_k)_{k \in \bN} \subset K$ such that
  \begin{equation*}
    \inf_{k \in \bN} \bE\big[ \max_{t \in [0,T]} |X_{n_k}(t,x_k)-X(t,x_k)|^2\big] \geq \delta.
  \end{equation*}
  Without loss of generality we assume that $(x_k)_{k \in \bN}$ is a sequence converging to some $x_0 \in K$. Since $X_n(0,x)=x$, $X(0,x)=x$, $M_n(0,x)=A_n(0,x) = 0$ for all $n \in \bN$ and the initial values $x_k$, $k \in \bN$, are in $K$ and using Lemma~\ref{lem:tightness} we can apply the Kolmogorov tightness criterion (see e.g. \cite[Problem~2.4.11]{Karatzas1991}) to obtain tightness of
  \begin{equation*}
    \bP_{X(\cdot,x_k), X_{n_k}(\cdot, x_k), A_{n_k}(\cdot, x_k), M_{n_k}(\cdot, x_k), B}
  \end{equation*}
  and by Prokhorov's theorem \cite[Theorem~2.4.7]{Karatzas1991} and the Skorokhod representation theorem \cite[Theorem~11.7.2]{Dudley2002} there is a probability space $(\hat{\Omega}, \hat{\mathcal{F}}, \hat{\bP})$ with continuous stochastic processes $\hat{X}^l, \hat{Y}^l, \hat{B}^l, \hat{A}^l, \hat{M}^l, l \in \bN, \hat{X}, \hat{Y}, \hat{B}, \hat{A}, \hat{M}$ such that
  \begin{equation*}
    \big(\hat{X}^l, \hat{Y}^l, \hat{B}^l, \hat{A}^l, \hat{M}^l\big) \stackrel{\mathscr{D}}{\sim} \big(X(\cdot,x_{k_l}), X_{n_{k_l}}(\cdot, x_{k_l}),B,A_{n_{k_l}}(\cdot, x_{k_l}), M_{n_{k_l}}(\cdot, x_{k_l})\big), \qquad l \in \bN,
  \end{equation*}
  and
  \begin{equation*}
    (\hat{X}^l, \hat{Y}^l, \hat{B}^l, \hat{A}^l, \hat{M}^l) \to (\hat{X}, \hat{Y}, \hat{B}, \hat{A}, \hat{M})
  \end{equation*}
  in $C([0,T]; \bR^d \times \bR^d \times \bR^m \times \bR^d \times \bR^d)$ as $l \to \infty$, $\hat{\bP}$-a.s.

  Using Fatou's lemma, we have
  \begin{align}\label{eq:ass}
    \begin{split}
    \delta & \leq \liminf_{k \to \infty} \bE\big[ \max_{t \in [0,T]} | X_{n_k}(t,x_k)-X(t,x_k)|^2\big]\\
    & \leq \liminf_{l \to \infty} \bE_{\hat{\bP}} \big[\max_{t \in [0,T]} |\hat{Y}^l_t-\hat{X}^l_t|^2\big]\\
    & \leq \bE_{\hat{\bP}} \big[\limsup_{l \to \infty} \max_{t \in [0,T]} |\hat{Y}^l_t-\hat{X}^l_t|^2\big]\\
    & = \bE_{\hat{\bP}} \big[\max_{t \in [0,T]} |\hat{Y}_t-\hat{X}_t|^2\big].
    \end{split}
  \end{align}

  Consider the measure $P$ on $(C([0,T];\bR^d),\mathcal{B}(C([0,T];\bR^d)))$ defined by $P = \hat{\bP} \circ \hat{Z}^{-1}$ where $\hat{Z}=\hat{A}+\hat{M}$ together with the filtration $(\mathcal{G}_t)_{t \in [0,T]}$ being the right-continuous augmentation of the completion of $(\sigma(\hat{Z}_s, 0 \leq s \leq t))_{t \in [0,T]}$. We shall show that $P$ solves the distribution-dependent Volterra local martingale problem~\ref{def:martproblem_new} given $(\delta_{x_0},b,\sigma,K_b, K_\sigma)$. Then, by \cite[Proposition~2.8]{Bergerhausen2026} $\hat{Y}$ is a solution to the SDE~\eqref{eq:MVSVE_2} w.r.t. the Brownian motion $\hat{B}$. The pathwise uniqueness then yields $\hat{Y} \equiv \hat{X}$ in contradiction to \eqref{eq:ass}, hence the theorem is correct.

  It is clear by construction that $\hat{Z}$ is a semimartingale. It is left to show the two equations in (ii):

  We introduce the stochastic processes $(Z^l)_{l \in \bN}$ and $(\hat{Z}^l)_{l \in \bN}$ by $Z^l := A_{n_{k_l}}(\cdot,x_{k_l}) + M_{n_{k_l}}(\cdot,x_{k_l})$ and $\hat{Z}^l := \hat{A}^l + \hat{M}^l$.

  \eqref{eq:X_MP_new}: Since $(\hat{Y}^l, \hat{M}^l) \stackrel{\mathscr{D}}{\sim}(X_{n_{k_l}}(\cdot,x_{k_l}),M_{n_{k_l}}(\cdot,x_{k_l}))$ for every $l \in \bN$ and since we have pathwise uniqueness by assumption we may use the result \cite[Theorem~1.5]{Kurtz2014} and consider $\hat{Y}^l$ as the output of a distribution-dependent stochastic Volterra equation to the input $\hat{M}^l$, i.e.
  \begin{equation*}
    \hat{Y}^l_t=x_{k_l} + \int_0^t K_{b,n_{k_l}}(s,t) b_{n_{k_l}}(s, \hat{Y}^l_s, \L(\hat{Y}^l_s)) \dd s + \int_0^t K_{\sigma,n_{k_l}}(s,t) \dd \hat{M}^l_s, \quad t \in [0,T], ~ \hat{\bP}\text{-a.s.}
  \end{equation*}

  We know that $x_{k_l} \to x_0$ and that $\hat{Y}^l \to \hat{Y}$ $\hat{\bP}$-a.s. By Lemma~\ref{lemma:convergence_integrals2} we get
  \begin{equation*}
    \Big( \int_0^t K_{b,n_{k_l}}(s,t) \dd \hat{A}^l_s \Big)_{t \in [0,T]} \xrightarrow{\hat{\bP}} \Big( \int_0^t K_b (s,t) \dd \hat{A}_s \Big)_{t \in [0,T]}.
  \end{equation*}
  To obtain
  \begin{equation}\label{eq:P_conv_mart_part}
    \Big( \int_0^t K_{\sigma,n_{k_l}}(s,t) \dd \hat{M}^l_s \Big)_{t \in [0,T]} \xrightarrow{\hat{\bP}} \Big( \int_0^t K_\sigma (s,t) \dd \hat{M}_s \Big)_{t \in [0,T]}
  \end{equation}
  we decompose
  \begin{align}\label{eq:decomp}
    \begin{split}
  	&\int_0^t K_{\sigma,n_{k_l}}(s,t) \dd \hat{M}^l_s - K_\sigma(s,t) \dd \hat{M}_s\\
  	&\quad= \int_0^t K_{\sigma,n_{k_l}}(s,t) \dd (\hat{M}^l_s - \hat{M}_s) + \int_0^t (K_{\sigma,n_{k_l}}(s,t)-K_\sigma(s,t)) \dd \hat{M}_s.
  	\end{split}
  \end{align}
  One can show uniform integrability of $((\hat{M}^l)^2)_{l \in \bN}$ in a similar way we use to show uniform integrability of $((\hat{Y}^l)^\eta)_{l \in \bN}$ below. Then, by the Vitali convergence theorem, see \cite[Exercise~3.4.13]{Cinlar2011}, $(\hat{M}^l)_{l \in \bN}$ converges towards $\hat{M}$ in $\mathcal{H}^2$. By the It{\^o} Isometry, for any $t \in  [0,T]$, we have
  \begin{equation}
  	\bE_{\hat{\bP}} \Big[ \Big( \int_0^t K_{\sigma,n_{k_l}} (s,t) \dd (\hat{M}^l_s - \hat{M}_s) \Big) ^2 \Big] = \bE_{\hat{\bP}} \Big[ \int_0^t (K_{\sigma,n_{k_l}} (s,t))^2 \dd \langle \hat{M}^l - \hat{M} \rangle_s \Big].
  \end{equation}
  By the $\mathcal{H}^2$-convergence of $(\hat{M}^l)_{l \in \bN}$, we have $\langle \hat{M}^l-{M}\rangle_T \to 0$ and since $t \mapsto \langle \hat{M}^l - \hat{M}\rangle_t$ is increasing and the $K_{\sigma,n_{k_l}}$ satisfy Assumption~\ref{ass:kernel} uniformly (note also Remark~\ref{rem:asskernel_to_assshort}) we can conclude that
  \begin{equation*}
  	\bE_{\hat{\bP}} \Big[ \Big( \int_0^t K_{\sigma,n_{k_l}} (s,t) \dd (\hat{M}^l_s - \hat{M}_s) \Big) ^2 \Big] \to 0, \qquad t \in [0,T].
  \end{equation*}
  This implies
  \begin{equation}\label{eq:decomp_conv1}
    \int_0^t K_{\sigma,n_{k_l}}(s,t) \dd \hat{M}^l_s \xrightarrow{\hat{\bP}} \int_0^t K_{\sigma,n_{k_l}} (s,t) \dd \hat{M}_s \text{ for all }t \in [0,T].
  \end{equation}
  
  Let $p > \frac{4+2\varepsilon}{\varepsilon}$. Then, denoting $\tilde{p}= \frac{p}{p-2}$, by the Burkholder--Davis--Gundy inequality and the H{\"o}lder inequality we have
  \begin{align*}
  	&\bE_{\hat{\bP}} \Big[\Big(\int_0^t (K_{\sigma,n_{k_l}}(s,t)-K_\sigma(s,t)) \dd \hat{M}_s\Big)^p \Big]^\frac{1}{p}\\
  	&\quad\leq C_{p,t} \Big( \int_0^t |K_{\sigma,n_{k_l}}(s,t)-K_\sigma(s,t)|^{2 \tilde{p}} \dd s \Big)^\frac{1}{2 \tilde{p}} \bE_{\hat{\bP}} \Big[ \int_0^t | \sigma(s,\hat{Y}_s, \mathcal{L}(\hat{Y}_s))|^p \dd s \Big]^\frac{1}{p}\Big].
  \end{align*}
  Note that by assumption
  \begin{equation}
  	\int_0^t |K_{\sigma,n_{k_l}}(s,t)-K_\sigma(s,t)|^{2 \tilde{p}} \dd s \to 0 \text{ for all } t \in [0,T],
  \end{equation}
  i.e.
  \begin{equation}\label{eq:decomp_conv2}
  	\bE_{\hat{\bP}} \Big[\Big(\int_0^t (K_{\sigma,n_{k_l}}(s,t)-K_\sigma(s,t)) \dd \hat{M}_s\Big)^p \Big]^\frac{1}{p} \to 0 \text{ for all } t \in [0,T].
  \end{equation}
  By combining \eqref{eq:decomp}, \eqref{eq:decomp_conv1} and \eqref{eq:decomp_conv2} we have a pointwise limit in probability. Knowing that the sequence converges $\hat{\bP}$-a.s. against a continuous process, by the uniqueness of limits this process must pointwise equal above limit, hence also indistinguishable from above limit. We get
  \begin{equation*}
    \Big(  \int_0^t K_{\sigma,n_{k_l}}(s,t) \dd \hat{M}^l_s \Big)_{t \in [0,T]} \to \Big(  \int_0^t K_{\sigma}(s,t) \dd \hat{M}_s \Big)_{t \in [0,T]}
  \end{equation*}
  $\hat{\bP}$-a.s. in $C([0,T];\bR^d)$. Hence, \eqref{eq:X_MP_new} holds $\hat{\bP}$-a.s. for $(\hat{Y}, \hat{Z})$.

  \eqref{def:M_f_new}: For $k \in \bN$ and $f \in C^2_0(\bR^d)$, we define the stochastic processes $(\mathcal{M}^{f,l}_t)_{t \in [0,T]}$, $l \in \bN$, by
  \begin{equation*}
    \mathcal{M}^{f,l}_t:= f(\hat{Z}^l_t)-\int_0^t \mathcal{A}^{f,l}(s,\hat{Y}^l_s,\L(\hat{Y}^l_s),\hat{Z}^l_s)\dd s,\quad t\in [0,T],
  \end{equation*}
  where
  \begin{equation*}
    \mathcal{A}^{f,l}(t,x,\rho,z) = b_{n_{k_l}}(t,x,\rho)^\top \nabla f(z) + \frac{1}{2} \mathrm{Tr}((\sigma_{n_{k_l}} \sigma_{n_{k_l}}^\top)(t,x,\rho)H_f(z)).
  \end{equation*}
  Due to $(\hat{Y}^l, \hat{Z}^l) \stackrel{\mathscr{D}}{\sim}(X_{n_{k_l}}(\cdot, x_{k_l}),Z^{l})$ and since $\bP \circ (Z^l)^{-1}$ solves the distribution-dependent Volterra local martingale problem given $(\delta_{x_{k_l}},b_{n_{k_l}},\sigma_{n_{k_l}}, K_{b,n_{k_l}},K_{\sigma,n_{k_l}})$ on $(\Omega, \mathcal F, \bP)$ by construction (see \cite[Proposition~2.8]{Bergerhausen2026}), it follows that $(\mathcal{M}^{f,l}_t)_{t \in [0,T]}$ is a local martingale on the filtered probability space $(C([0,T];\bR^d),\mathcal{B}(C([0,T];\bR^d)), (\mathcal{G}_t)_{t\in[0,T]},P)$ for every $l \in \bN$.

  Next we show that $\lim_{l \to \infty} W_\eta(\L(\hat{Y}^l),\L(\hat{Y})) = 0$. By \cite[Theorem~5.5]{Carmona2018} for this to hold we need the weak convergence of $(\L(\hat{Y}^l))_{l \in \bN}$ towards $\L(\hat{Y})$, which is already implied by the above, and uniform integrability in the sense of
  \begin{equation*}
    \lim_{r \to \infty} \sup_{l \in \bN} \bE_{\hat{\bP}} \big[ \| \hat{Y}^l-y\|_\infty^\eta \mathbbm{1}_{\{\|\hat{Y}^l-y\|_\infty \geq r\}} \big]=0
  \end{equation*}
  for some $y \in C([0,T];\bR^d)$. For the latter we show that there is a uniform bound on $\bE_{\hat{\bP}} \big[\sup_{t \in [0,T]}|\hat{Y}^l_t|^\eta\big]$. We have
  \begin{align*}
    \bE_{\hat{\bP}} \Big[\sup_{t \in [0,T]}|\hat{Y}^l_t|^\eta\Big] & \leq C_\eta \Big(x_{k_l}^\eta + \bE_{\hat{\bP}} \Big[ \sup_{t \in [0,T]} |\hat{Y}^l_t-x_{k_l}|^\eta \Big] \Big)\\
    & \leq C_\eta \Big(x_{k_l}^\eta + \bE_{\hat{\bP}} \Big[ \sup_{0 \leq s,t \leq T} |\hat{Y}^l_t-\hat{Y}^l_s|^\eta \Big] \Big ).
  \end{align*}
  Recall that the initial values $x_{k_l}$ remain in a compactum. $\hat{Y}^l$ is $h$-H{\"o}lder for all $h \in (0,\gamma)$ by Lemma~\ref{lem:tightness}. Now choose some $\beta \in (0, \gamma-\frac{1}{p})$, $\frac{2}{\beta}< \alpha$ and obtain by the Garsia--Rodemich--Rumsey inequality (see \cite[Theorem~2.1.3]{Stroock1979}) that
  \begin{align*}
    &\sup_{0 \leq s,t \leq T} |\hat{Y}^l_t - \hat{Y}^l_s|\\
    &\quad \leq 8 \int_0^T \Bigg(4 \int_0^T \int_0^T \Big(\frac{|\hat{Y}^l_t-\hat{Y}^l_s|}{|t-s|^\beta}\Big)^\alpha \dd s \dd t\, u^{-2}\Bigg)^{\frac{1}{\alpha}} \dd u^\beta\\
    &\quad = 8 ~ 4^\frac{1}{\alpha} \Bigg( \int_0^T \frac{1}{u^\frac{2}{\alpha}} \dd u^\beta \Bigg) \Bigg(\int_0^T \int_0^T \Bigg(\frac{|\hat{Y}^l_t-\hat{Y}^l_s|}{|t-s|^\beta}\Bigg)^\alpha \dd s \dd t\Bigg)^\frac{1}{\alpha}\\
    &\quad = 8 ~ 4^\frac{1}{\alpha} \Bigg( \int_0^T \beta u^{\beta-1-\frac{2}{\alpha}} \dd u \Bigg) \Bigg(\int_0^T \int_0^T \Big(\frac{|\hat{Y}^l_t-\hat{Y}^l_s|}{|t-s|^\beta}\Big)^\alpha \dd s \dd t\Bigg)^\frac{1}{\alpha}.
  \end{align*}
  The former integral is finite since $\beta -1-\frac{2}{\alpha} >-1$. By Lemma~\ref{lem:tightness} there is a $C > 0$ such that for all $l \in \bN$
  \begin{equation*}
    \bE_{\hat{\bP}}\big[|\hat{Y}^l_t-\hat{Y}^l_s|^{\eta \alpha} \big] \leq C |t-s|^{\eta \alpha \gamma}, \qquad s,t \in [0,T].
  \end{equation*}
  Hence, by the Jensen inequality and the linearity of integrals, we have
  \begin{align*}
    \bE_{\hat{\bP}}\big[\sup_{0 \leq s,t \leq T} |\hat{Y}^l_t-\hat{Y}^l_s|^\eta\big] & \leq C \bE_{\hat{\bP}}\Big[\Big(\int_0^T \int_0^T \frac{|\hat{Y}^l_t-\hat{Y}^l_s|^\alpha}{|t-s|^{\alpha \beta}} \dd s \dd t\Big)^\frac{\eta}{\alpha}\Big]\\
    & \leq C T^{\frac{2(\eta-1)}{\alpha}}\bE_{\hat{\bP}}\Big[\Big(\int_0^T \int_0^T \frac{|\hat{Y}^l_t-\hat{Y}^l_s|^{\eta \alpha}}{|t-s|^{\eta \alpha \beta}} \dd s \dd t\Big)^\frac{1}{\alpha}\Big]\\
    & \leq C_{T,\eta,\alpha} \Big(\int_0^T \int_0^T \frac{\bE_{\hat{\bP}}[|\hat{Y}^l_t-\hat{Y}^l_s|^{\eta \alpha}]}{|t-s|^{\eta \alpha \beta}} \dd s \dd t\Big)^\frac{1}{\alpha}\\
    & \leq C_{T,\eta,\alpha} \Big(\int_0^T \int_0^T \frac{C|t-s|^{\eta \gamma \alpha}}{|t-s|^{\eta \alpha \beta}} \dd s \dd t\Big)^\frac{1}{\alpha}\\
    & \leq \tilde{C} \Big(\int_0^T \int_0^T |t-s|^{\eta \alpha(\gamma-\beta)} \dd s \dd t\Big)^\frac{1}{\alpha}.
  \end{align*}
  Note that since $\eta \alpha (\gamma - \beta) > 0$, the integral is finite. This shows that there is a $C>0$ s.t. for all $l \in \bN$ we have $\bE_{\hat{\bP}}\big[\sup_{t \in [0,T]} | \hat{Y}^l_t|^\eta \big] \leq C$ and the convergence of the Wasserstein distance of the laws follows.

  Now, Lemma~\ref{lemma:convergence_integrals2} implies that $\mathcal{M}^{f,l} \to \mathcal{M}^{f}$ weakly as $l \to \infty$. Then by \cite[Proposition~IX.1.17]{Jacod2003} the limiting process $(\mathcal{M}^f_t)_{t \in [0,T]}$ is a local $\hat{\bP}$-martingale.
\end{proof}

\section{Stability of solutions to the Volterra local martingale problem} \label{subsec:weak_results}

In this section, we analyse the stability of solutions to Volterra local martingale problems (see Definition~\ref{def:martproblem_new}) associated with the distribution-dependent stochastic Volterra equation~\eqref{eq:MFSVE} with respect to their coefficients. The following theorem is the main result of this section.

\begin{theorem}\label{theorem:approximation_weak}
  Suppose that
  \begin{itemize}
    \item Assumption~\ref{ass:volatility} holds,
    \item all pairs of kernels $(K_{b,n}, K_{\sigma,n})$ $n \in \bN$ and $(K_b,K_{\sigma})$ satisfy Assumption~\ref{ass:kernel} using the same constants $\gamma, \epsilon, L$,
    \item $\mu_{0,n} \in \mathcal{P}_{p}(\R^d)$ for all $n \in \bN$ holds for some $p > \max\{\frac{2 \eta + 1}{\gamma},2+\frac{4}{\varepsilon}\}$ with uniformly bounded $p$-moments, meaning $\sup_{n \in \bN} \int_{\bR^d} |x|^p \mu_{0,n}(\d x) < \infty$,
	\item $\mu_{0,n}$ converges to $\mu$ as $n \to \infty$,
	\item for any compact subsets $K \subset \bR^d$ and $\mathcal K \subset \mathcal P_\eta(\bR^d)$,
	\begin{equation*}
	  \lim_{n \to \infty} \sup_{t \in [0,T]} \sup_{x \in K} \sup_{\mu \in \mathcal K} |\sigma_n(t,x,\mu)-\sigma(t,x,\mu)|+|b_n(t,x,\mu)-b(t,x,\mu)|=0
	\end{equation*}
	Additionally
	\begin{equation*}
	  \sup_{t \in [0,T]} \int_0^t |K_{b,n}(s,t)-K_b(s,t)|^{1+\epsilon} \dd s \to 0 \text{ as } n \to \infty
	\end{equation*}
	and
	\begin{equation*}
	  \sup_{t \in [0,T]} \int_0^t |K_{\sigma,n}(s,t)-K_\sigma(s,t)|^{2+\epsilon} \dd s \to 0 \text{ as } n \to \infty.
	\end{equation*}
  \end{itemize}
  Furthermore, suppose that the distribution-dependent Volterra local martingale problem given $(\mu_{0,n},K_{b,n}, K_{\sigma,n}, b_n, \sigma_n)$, $n \in \bN$, has a solution $P_n \in \mathcal{P}_p(C([0,T];  \bR^d))$ for each $n \in \bN$ and that the distribution-dependent Volterra local martingale problem given $(\mu, K_{b}, K_{\sigma}, b, \sigma)$ has a unique (in law) solution $P$. Then, $P_n$ converges weakly to $P$ as $n \to \infty$.
\end{theorem}

\begin{remark}
  In the case of an ordinary stochastic Volterra equation, i.e. no dependency on the law, the initial condition can be relaxed to $\mu_0 \in \mathcal{P}_{p}(\R^d)$ for some $p > \max\{\frac{1}{\gamma},\frac{4+2\varepsilon}{\varepsilon}\}$.
\end{remark}

Denote with $P_n \in \mathcal{P}_p(C([0,T];\bR^d))$ a solution to the distribution-dependent Volterra local martingale problem associated with $(\mu_{0,n},K_{b,n}, K_{\sigma,n}, b_n, \sigma_n)$ in the sense of Definition~\ref{def:martproblem_new}. Likewise, let $P$ denote the solution to the distribution-dependent Volterra local martingale problem associated with $(\mu, K_{b}, K_{\sigma}, b, \sigma)$.

We work on the canonical path space
\begin{equation*}
  (\Omega,\mathcal F) := (C([0,T];\bR^d),\mathcal B(C([0,T];\bR^d))),
\end{equation*}
and denote by $Z(t)(\omega) := \omega(t)$, $t \in [0,T]$, $\omega \in \Omega$, the canonical coordinate process under $P_n$. For each $n \in \bN$ we consider the filtered probability space
\begin{equation*}
  (\Omega, \mathcal F, (\mathcal{F}^n_t)_{t \in [0,T]},P_n), \qquad \mathcal{F}^n_t := \mathcal G_{t+}^n,
\end{equation*}
where $(\mathcal G_t^n)_{t \in [0,T]}$ is the $P_n$-augmentation of the canonical filtration $(\sigma(Z_s \colon s \leq t))_{t \in [0,T]}$. By Definition~\ref{def:martproblem_new}, under $P_n$ the canonical process $Z$ is a continuous $d$-dimensional semimartingale with $Z_0 \sim \mu_{0,n}$. We introduce $Z_n=Z$, if $Z$ is the canonical process under $P_n$. This is for notational convenience especially in the formulation of the martingale problem. Then, there exists a decomposition
\begin{equation*}
  Z_n = A_n + M_n,
\end{equation*}
where $A_n = (A_n(t))_{t \in [0,T]}$ is $(\mathcal{F}^n_t)$-predictable of finite variation and $M_n=(M_n(t))_{t \in [0,T]}$ is a continuous $(\mathcal{F}^n_t)$-local martingale starting in $0$.

We denote by $X_n=(X_n(t))_{t \in [0,T]}$ the Volterra state process obtained from $(Z_n,A_n,M_n)$ via the transform \eqref{eq:X_MP_new} by using $K_{b,n}$ and $K_{\sigma,n}$:
\begin{equation}
  X_n(t) := Z_n(0) + \int_0^t K_{b,n}(s,t) \dd A_n(s) + \int_0^t K_{\sigma,n}(s,t) \dd M_n(s), \qquad t \in [0,T], \, P_n\text{-a.s.}
\end{equation}
With $\bE_n$ we denote the (conditional) expectation w.r.t. $P_n$.

Analogously, on $(\Omega, \mathcal F, (\mathcal{F}_t)_{t \in [0,T]},P)$ with $\mathcal{F}_t := \mathcal{G}_{t+}$ defined from the $P$-augmentation, we can write $Z(t)(\omega) := \omega(t)$ for the canonical process under $P$, choose a semimartingale decomposition $Z = A+M$ and define $X=(X(t))_{t \in [0,T]}$ by
\begin{equation}
  X(t) := Z(0) + \int_0^t K_{b}(s,t) \dd A(s) + \int_0^t K_{\sigma}(s,t) \dd M(s), \qquad t \in [0,T], \, P\text{-a.s.}
\end{equation}

Equivalently to the previous section we present three lemmas beforehand. By \cite[Lemma~2.3~(ii) and Proposition~2.8]{Bergerhausen2026} a solution to the martingale problem can be realized as a weak solution of the DDSVE, i.e. for each $n \in \mathbb{N}$ there is a Brownian motion $B_n$ on some possibly extended filtered probability space, such that
\begin{equation*}
  M_n(t) = \int_0^t \sigma_n(s,X_n(s),\mathcal{L}(X_n(s))) \dd B_n(s), \qquad t \in [0,T].
\end{equation*}
Hence, the proofs of Lemma~\ref{lemma:pathcontinuity}, \ref{lemma:boundedness2} and \ref{lem:tightness} directly translate to this setting:

\begin{lemma}\label{lemma:pathcontinuity_weak}
  Under the assumptions of Theorem~\ref{theorem:approximation_weak} $X_n$ is $\bP_n$-a.s. continuous for all  $n \in \bN$.
\end{lemma}

\begin{lemma}\label{lemma:boundedness2_weak}
  Under the assumptions of Theorem~\ref{theorem:approximation_weak} for every $q \in [1,p]$ and every compact set $K \subset \bR^d$ there is a $C_q >0$ such that $\bE_n[|X_n(t)|^q] \leq C_q$ for all $n \in \bN$, $t \in [0,T]$.
\end{lemma}

\begin{proof}
  The proof follows the steps of the proof of Lemma~\ref{lemma:boundedness2}. One only needs to put a $\mu_{0,n}$-distributed random variable in the place of the deterministic initial condition there and take $q$ not larger than the finite moments of this mentioned random variable.
\end{proof}

\begin{lemma}\label{lem:tightness_weak}
  Under the assumptions of Theorem~\ref{theorem:approximation_weak} for all compact $K \subset \bR^d$ there is an $C > 0$ such that
  \begin{align*}
    \bE_n[|X_n(t)-X_n(t')|^p] & \leq C|t'-t|^{\gamma p},\\
    \bE_n[|M_n(t)-M_n(t')|^p] & \leq C|t'-t|^{\gamma p},\\
    \bE_n[|A_n(t)-A_n(t')|^p] & \leq C|t'-t|^{\gamma p},
  \end{align*}
  for all $n \in \bN$, $x \in K$, $t,t' \in [0,T]$.
\end{lemma}

\begin{proof}
  Since the linear growth bounds in \eqref{eq:lin_grow_bound} are assumed to hold uniformly and the kernels are assumed to satisfy Assumption~\ref{ass:kernel} with the same constants, using the uniform bound in Lemma~\ref{lemma:boundedness2} the lemma can be proven following the exact same steps as in the proof of \cite[Lemma~3.8]{Bergerhausen2026}.
\end{proof}

With the previous lemmata at hand we can prove the Theorem~\ref{theorem:approximation_weak}.

\begin{proof}[Proof of Theorem~\ref{theorem:approximation_weak}]
  We proceed with a proof by contradiction. To that end, we assume that $(Z_n)_{n \in \bN}$ does not converge weakly to $Z$ as $n \to \infty$. Then, there is a subsequence $(Z_{n_k})_{k \in \bN}$ that has no further subsequence that converges weakly to $Z$.

  We define for all $n \in \bN$
  \begin{align}
    X_n(t) &:= Z_n(0)+\int_0^t K_{b,n}(s,t)\dd A_n(s)+\int_0^t K_{\sigma,n}(s,t) \dd M_n(s), && t \in [0,T]. \label{eq:Xk_def}
  \end{align}
  By the assumption on the initial conditions stated in the theorem and using Lemma~\ref{lem:tightness_weak} we can apply the Kolmogorov tightness criterion (see e.g. \cite[Problem~2.4.11]{Karatzas1991}) to obtain tightness of $P_{A_{n_k}, M_{n_k}, X_{n_k}}$ and by Prokhorov's theorem \cite[Theorem~2.4.7]{Karatzas1991} there is a weakly convergent subsequence $P_{A_{n_{k_l}}, M_{n_{k_l}}, X_{n_{k_l}}}$ converging to some $P_{A, M, X}$. We shall show that $P_{A + M}$, i.e. the law of the process $Z=A+M$ on $C([0,T];\bR^d)$, solves the distribution-dependent Volterra local martingale problem given $(\mu, b, \sigma, K_b, K_\sigma)$.

  First we show (ii) of Definition~\ref{def:martproblem_new}: Let $(n_{k_l})_{l \in \bN}$ be the subsequence fixed above such that $P_{A_{n_{k_l}},M_{n_{k_l}},X_{n_{k_l}}} \rightarrow P_{A,M,X}$ weakly on $C([0,T];\R^d\times\R^d\times \R^d)$. By the Skorokhod representation theorem (see e.g. \cite[Theorem~11.7.2]{Dudley2002}), there exists a probability space $(\hat\Omega,\hat{\mathcal{F}},\hat P)$ and continuous processes $\hat A^l,\hat M^l,\hat X^l$, $l\in\mathbb N$, and $\hat A,\hat M,\hat X$ such that
  \begin{align*}
    (\hat A^l,\hat M^l,\hat X^l) &\stackrel{\mathscr{D}}{\sim} (A_{n_{k_l}},M_{n_{k_l}},X_{n_{k_l}}),\qquad l \in \bN,\\
    (\hat A,\hat M,\hat X) &\stackrel{\mathscr{D}}{\sim} (A,M,X),\\
    (\hat A^l,\hat M^l,\hat X^l) &\to (\hat A,\hat M,\hat X) \quad \text{ in } C([0,T];\bR^d \times \bR^d \times \bR^d) \text{ as } l \to \infty,\ \hat P \text{-a.s.}
  \end{align*}
  Define $\hat Z^l:=\hat A^l+\hat M^l$ and $\hat Z:=\hat A+\hat M$. Note that we have by \cite[Lemma~2.3]{Bergerhausen2026}
  \begin{align*}
    \hat A^l_t -\hat A^l_0 &= \int_0^t b(s,\hat{X}^l_s,\L(\hat{X}^l_s))\dd s,\\
    \hat M^l_t -\hat M^l_0 &= \int_0^t \sigma(s,\hat{X}^l_s,\L(\hat{X}^l_s))\dd \hat{B}^l_s,
  \end{align*}
  for all $l \in \bN$ and $t \in [0,T]$ and some $\hat{P}$-Brownian motions $\hat{B}^l$, $l \in \bN$. By Lemma~\ref{lemma:app_space_change}
  \begin{equation*}
    \hat{X}^l_t = \hat{Z}^l_0+ \int_0^t K_{b,n_{k_l}}(s,t) \dd \hat{A}^l_s + \int_0^t K_{\sigma,n_{k_l}}(s,t) \dd \hat{M}^l_s, \qquad t \in [0,T],
  \end{equation*}
  for all $l \in \bN$.

  We show that
  \begin{equation*}
    \hat{Z}^l_0+ \int_0^t K_{b,n_{k_l}}(s,t) \dd \hat{A}^l_s + \int_0^t K_{\sigma,n_{k_l}}(s,t) \dd \hat{M}^l_s
  \end{equation*}
  converges in probability to
  \begin{equation*}
    \hat{Z}_0+ \int_0^t K_{b}(s,t) \dd \hat{A}_s + \int_0^t K_{\sigma,n}(s,t) \dd \hat{M}_s
  \end{equation*}
  as $l \to \infty$, to conclude that
  \begin{equation}\label{eq:X_def}
    \hat X_t = \hat Z_0+\int_0^t K_b(s,t) \dd \hat A_s + \int_0^t K_\sigma(s,t) \dd \hat M_s,\qquad t\in[0,T], \,\hat{P}\text{-a.s.}
  \end{equation}

  By Lemma~\ref{lemma:convergence_integrals2},
  \begin{equation}\label{eq:FV_conv}
    \Big(\int_0^t K_{b,n_{k_l}}(s,t) \dd \hat A^l_s \Big)_{t \in [0,T]} \xrightarrow{\hat P} \Big(\int_0^t K_b(s,t)\dd \hat A_s \Big)_{t \in [0,T]}
  \end{equation}
  in $(C([0,T];\R^d),\|\cdot\|_\infty)$.

  For the martingale part, decompose for each fixed $t \in [0,T]$,
  \begin{align}\label{eq:mart_decomp}
    \begin{split}
	&\int_0^t K_{\sigma,n_{k_l}}(s,t)\dd \hat M^l_s - \int_0^t K_\sigma(s,t) \dd \hat M_s \\
	&\quad = \int_0^t K_{\sigma,n_{k_l}}(s,t) \dd (\hat M^l_s-\hat M_s) + \int_0^t (K_{\sigma,n_{k_l}} - K_\sigma)(s,t) \dd \hat M_s.
	\end{split}
  \end{align}
  Uniform integrability follows as in the proof of Theorem~\ref{theorem:approximation}. Then, by using the fact that $\langle \hat M_l - \hat M\rangle_T \to 0$ in probability, we obtain $\hat M_l \to \hat M$ in $\mathcal H^2$ along a subsequence by the Vitali convergence theorem, see \cite[Exercise~3.4.13]{Cinlar2011}. Consequently, for each fixed $t \in [0,T]$,
  \begin{equation}\label{eq:first_mart_term}
    \hat\bE \Big[ \Big| \int_0^t K_{\sigma,n_{k_l}}(s,t)\dd (\hat M^l_s - \hat M_s) \Big|^2\Big] = \hat\bE \Big[ \int_0^t |K_{\sigma,n_{k_l}}(s,t)|^2 \dd \langle \hat M^l - \hat M \rangle_s \Big] \to 0,
  \end{equation}
  where we used It\^{o}'s Isometry and the uniform $L^2$-boundedness of $K_{\sigma,n_{k_l}}(\cdot,t)$ uniformly in $l$ (see Assumption~\ref{ass:kernel}).

  For the second term in \eqref{eq:mart_decomp}, the Burkholder--Davis--Gundy inequality yields for any fixed $t \in [0,T]$ with $\tilde p = \frac{p}{p-2}$,
  \begin{align}\label{eq:second_mart_term}
    \begin{split}
    &\hat\bE \Big[ \Big| \int_0^t (K_{\sigma,n_{k_l}}-K_\sigma)(s,t) \dd \hat M_s\Big|^p\Big]^{\frac{1}{p}} \\
    &\quad\leq C_{p,t}\Big( \int_0^t |K_{\sigma,n_{k_l}}(s,t) - K_\sigma(s,t)|^{2\tilde p}\dd s \Big)^{\frac{1}{2\tilde p}} \cdot \hat\bE \Big[ \int_0^t |\sigma(s,\hat X(s),\mathcal{L} (\hat X(s)))|^p \dd s\Big]^{\frac{1}{p}}.
    \end{split}
  \end{align}
  Since $\tilde{p} \leq 1 + \frac{\varepsilon}{2}$, the kernel convergence assumption implies that the first factor tends to $0$ as $l \to \infty$. The second factor is finite uniformly by the linear growth bound and Lemma~\ref{lemma:boundedness2_weak}. Therefore, the right-hand side tends to $0$.

  Combining \eqref{eq:FV_conv}, \eqref{eq:mart_decomp}-\eqref{eq:second_mart_term} and the convergence $\hat Z^l_0 \to\hat Z_0$, we obtain \eqref{eq:X_def} in probability in $C([0,T];\R^d)$.

  Fix some $f \in C_0^2(\bR^d)$ and define on $(\hat{\Omega},\hat{\mathcal{F}},\hat P)$ the processes
  \begin{align*}
    \hat{\mathcal{M}}^{f,l}_t &:= f(\hat Z^l_t) - \int_0^t \mathcal{A}^{f,l}(s,\hat X^l_s,\mathcal{L}(\hat X^l_s),\hat Z^l_s)\dd s,\\
    \mathcal{A}^{f,l}(t,x,\rho,z)&:= b_{n_{k_l}}(t,x,\rho)^\top \nabla f(z)+\frac{1}{2}\mathrm{Tr}\Big((\sigma_{n_{k_l}}\sigma_{n_{k_l}}^\top)(t,x,\rho)\,H_f(z)\Big),
  \end{align*}
  where $\mathrm{Tr}$ denotes the trace and $H_f$ the Hessian of $f$, and analogously
  \begin{equation*}
    \hat{\mathcal{M}}^{f}_t := f(\hat Z_t)-\int_0^t \mathcal{A}^{f}(s,\hat X_s,\mathcal{L} (\hat X_s),\hat Z_s) \dd s,
  \end{equation*}
  with $\mathcal{A}^f$ defined from $(b,\sigma)$ as in \eqref{eq:operator A}.

  As in the proof of Theorem~\ref{theorem:approximation}, using the Garsia--Rodemich--Rumsey inequality, one can show
  \begin{equation*}
    \lim_{r \to \infty} \sup_{l \in \bN} \hat{\bE} \big[ \| \hat{X}^l-y\|_\infty^\eta \mathbbm{1}_{\{\|\hat{X}^l-y\|_\infty \geq r\}} \big]=0.
  \end{equation*}
  By \cite[Theorem~5.5]{Carmona2018} this implies $W_\eta(\mathcal{L}(\hat X^l),\mathcal{L}(\hat X))\to 0$. Moreover, by the construction of the Skorokhod representation,
  \begin{equation*}
    ( \hat{X}^l, \hat{Z}^l) \to ( \hat{X}, \hat{Z}) \qquad C([0,T];\bR^d \times \bR^d),\, \hat{P}\text{-a.s.}
  \end{equation*}
  Since $b_{n_{k_l}} \to b$ and $\sigma_{n_{k_l}} \to \sigma$ locally uniformly and $f \in C^2_0(\bR^d)$, the functions
  \begin{equation*}
    \mathcal{A}^{f,l}(t,x,\rho,z):= b_{n_{k_l}}(t,x,\rho)^\top \nabla f(z) +\frac{1}{2} \text{Tr}\big( (\sigma_{n_{k_l}}\sigma_{n_{k_l}}^\top)(t,x,\rho) H_f (z)\big),
  \end{equation*}
  converge locally to
  \begin{equation*}
    \mathcal{A}^{f}(t,x,\rho,z):= b(t,x,\rho)^\top \nabla f(z) +\frac{1}{2} \text{Tr}\big( (\sigma \sigma^\top)(t,x,\rho) H_f (z)\big)
  \end{equation*}
  as $l \to \infty$. Furthermore, by the linear growth condition on $b_n$ and $\sigma_n$, $n \in \bN$, and the boundedness of $\nabla f$ and $H_f$, the family $\mathcal{A}^{f,l}$ satisfies a uniform quadratic growth bound. Therefore, by Lemma~\ref{lemma:convergence_integrals2} applied with $K \equiv 1$, we conclude that
  \begin{equation}\label{eq:drift_conv_weak}
    \int_0^\cdot \mathcal{A}^{f,l}(s,\hat X^l_s,\mathcal{L}(\hat X^l_s),\hat Z^l_s)\dd s \xrightarrow{\hat{P}} \int_0^\cdot \mathcal{A}^{f}(s,\hat X_s,\mathcal{L}(\hat X_s),\hat Z_s)\dd s
  \end{equation}
  in $C([0,T];\bR)$. Together with $\hat Z^l \to \hat Z$ in probability, this implies $\hat{\mathcal{M}}^{f,l} \to \hat{\mathcal{M}}^{f}$ weakly in $C([0,T];\bR)$.

  For each $l \in \bN$, since $P_{n_{k_l}}$ solves the martingale problem, $\hat{\mathcal{M}}^{f,l}$ is a local martingale (with respect to the right-continuous augmentation of the canonical filtration generated by $\hat Z^l$). By \cite[Proposition~IX.1.17]{Jacod2003}, \eqref{eq:drift_conv_weak} yields that $\hat{\mathcal{M}}^{f}$ is a local martingale with respect to the filtration generated by $\hat Z$.

  Therefore, the law $\hat P \circ \hat Z^{-1}$ solves (ii) of Definition~\ref{def:martproblem_new} for $(\mu_0,b,\sigma,K_b,K_\sigma)$. By uniqueness in law of solutions to the limiting martingale problem, we conclude that $\hat P \circ \hat Z^{-1} = P$. Hence, $P_{n_{k_l}} \to P$ as $l \to \infty$, which contradicts the initial assumption of $(Z_{n_k})_{k \in \bN}$ being a sequence that has no further subsequence that converges weakly to $Z$.

  Clearly (i) of Definition~\ref{def:martproblem_new} holds, as $\hat{Z}=\hat{A}+\hat{M}=\lim_{l\to\infty}\hat{A}^l+\hat{M}^l$ is a semimartingale and $\mathcal{L}(\hat{Z}^l_0)=\mu_{0,n_l} \to \mu$ as $l \to \infty$ by assumption.
\end{proof}

\appendix
\section{Auxiliary results on Volterra type integrals}\label{sec: appendix}

This appendix contains two auxiliary results regarding Volterra type integrals.

\begin{lemma}\label{lemma:convergence_integrals2}
  Let $f \colon [0,T] \times \bR^d \times \mathcal{P}_\eta(\bR^d) \to \bR^e$ be a function such that for all compact sets $\mathcal K \subset \bR^d$, $\bar{\mathcal K}\subset \mathcal{P}_\eta(\bR^d)$ and every $\epsilon >0$ there exists a $\delta >0$ such that
  \begin{equation*}
    |f(t,x,\mu)-f(t,y,\nu)|\leq \epsilon
  \end{equation*}
  for all $t \in [0,T]$, $x,y \in \mathcal K$ satisfying $|x-y| \leq \delta$, and $\mu, \nu \in \bar{\mathcal K}$ with $W_\eta(\mu,\nu)\leq \delta$. Let $(f_k)_{k \in \bN}$ be a sequence of functions such that $f_k \colon [0,T] \times \bR^d \times \mathcal{P}_\eta(\bR^d) \to \bR^e$ and
  \begin{equation*}
    |f(t,x,\mu)|+|f_k(t,x,\mu)|\leq C(1+|x|^2+W_\eta(\mu,\delta_0)^2), \quad x \in \bR^d, \, t \in [0,T], \,\mu \in \mathcal P_\eta(\bR^d),
  \end{equation*}
  for all $k \in \bN$ and for some $C>0$, and $f_k \to f$ locally uniformly.

  Let $K, K_k \colon \Delta_T \to \bR$, $k \in \bN$, be measurable functions and let $K$ be bounded in $L^1([0,T])$ uniformly in the second variable, i.e. $\sup_{t \in [0,T]} \int_0^t |K(s,t)| \dd s \leq M$ for some $M > 0$, and
  \begin{equation*}
    \sup_{t \in [0,T]} \int_0^t |K_k(s,t)-K(s,t)| \dd s \to 0 \quad\text{as } k \to \infty.
  \end{equation*}
  If $(X^k)_{k \in \bN}$ is a sequence of continuous stochastic processes such that $X^k \to X$ in $C([0,T]; \bR^d)$ as $k \to \infty$ $\bP$-a.s. and $(\L(X^k))_{k \in \bN} \subset \mathcal{P}_\eta(C([0,T];\bR^d))$, $\L(X)\in \mathcal{P}_\eta(C([0,T];\bR^d))$ such that $W_\eta(\L(X^k),\L(X)) \to 0$ as $k \to \infty$, then
  \begin{equation*}
    \int_0^{\cdot} K_k(s,\cdot) f_k(s,X^k_s,\L(X^k_s))\dd s \xrightarrow{\bP} \int_0^{\cdot} K(s,\cdot) f(s,X_s,\L(X_s))\dd s
  \end{equation*}
  with respect to $\|\cdot\|_\infty$ as $k \to \infty$.
\end{lemma}

\begin{proof}
  Let $\epsilon, \theta >0$ be arbitrary but fixed. Choose $N>W_\eta(\delta_0,\L(X))$ and $C_1 \in \bN$, such that
  \begin{equation*}
    \bP\Big(\|X\|_\infty \geq \frac{N}{2}\Big) \leq \frac{\theta}{4}
    \quad\text{and}\quad
    \bP\Big(\|X^k-X\|_\infty \geq \frac{N}{2}\Big) \leq \frac{\theta}{4}
  \end{equation*}
  for all $k \geq C_1$. Then, we observe that
  \begin{align*}
    \bP(\|X\|_\infty \vee \|X^k\|_\infty \geq N) & \leq \bP(\{\|X\|_\infty \geq N\} \cup \{\|X^k-X\|_\infty + \|X\|_\infty \geq N\})\\
    & \leq \bP\Big(\|X\|_\infty \geq \frac{N}{2}\Big) + \bP \Big(\|X^k-X\|_\infty \geq \frac{N}{2}\Big)\\
    & \leq \frac{\theta}{2}
  \end{align*}
  and for all $k \geq C_1$. By the triangle inequality
  \begin{equation*}
    W_\eta(\delta_0,\L(X^k_s)) \leq W_\eta(\delta_0,\L(X_s)) + W_\eta(\L(X_s),\L(X^k_s)),
  \end{equation*}
  by $W_\eta(\L(X^k),\L(X)) \to 0$ as $k \to \infty$ and $N>W_\eta(\delta_0,\L(X))$, there is a $C_2 \in \bN$ such that $W_\eta(\delta_0,\L(X^k_s)) \leq N$ for all $k \geq C_2$. Hence, for every $t \in [0,T]$ on $\{\|X\|_\infty \vee \|X^k\|_\infty \leq N\}$ we have
  \begin{align*}
    & |G^k_t - G_t|\\
    &\quad:= \Big| \int_0^t K_k(s,t)f_k(s,X^k_s, \L(X^k_s))\dd s- \int_0^t K(s,t) f(s,X_s, \L(X_s)) \dd s  \Big|\\
    &\quad \leq \Big|\int_0^t (K_k(s,t)-K(s,t)) f_k(s, X^k_s, \L(X^k_s)) \dd s \Big|\\
    &\quad \qquad + M \Big(\sup_{s \in [0,T]} \sup_{x \in [-N,N]} \sup_{\mu \colon W_\eta(\delta_0,\mu) \leq N} |f_k(s,x,\mu)-f(s,x,\mu)| \\
    &\quad \qquad \qquad + \sup_{s \in [0,T]} |f(s,X^k_s, \L(X^k_s))-f(s,X_s,\L(X_s))| \Big).
  \end{align*}
  By the growth condition we obtain
  \begin{equation*}
    \Big|\int_0^t (K_k(s,t)-K(s,t)) f_k(s,X^k_s, \L(X^k_s))\dd s \Big| \leq \int_0^t |K_k(s,t) - K(s,t)| \dd s ~ C(1+2 N^2).
  \end{equation*}
  Therefore, by the convergence of the sequence of kernels there is a $C_3 \in \bN$ such that
  \begin{equation*}
    \Big|\int_0^t (K_k(s,t)-K(s,t)) f_k(s,X^k_s, \L(X^k_s))\dd s\Big| \leq \frac{\epsilon}{3}, \quad\text{for all } k \geq C_2.
  \end{equation*}
  We choose $C_4 \in \bN$ sufficiently large such that
  \begin{equation*}
    \sup_{s \in [0,T]} \sup_{x \in [-N,N]} \sup_{\mu \colon W_\eta ( \delta_0,\mu) \leq N} |f_k(s,x,\mu)-f(s,x,\mu)| \leq \frac{\epsilon}{3M}.
  \end{equation*}
  Because of the continuity condition, there exists some continuous non-decreasing function $g \colon [0,\infty)^2 \to [0,\infty)$ with $g(0,0)=0$, such that for all $x,y\in [-N,N]$ we have
  \begin{equation*}
    |f(t,x,\mu)-f(t,y,\nu)| \leq g(|x-y|,W_\eta(\mu,\nu)), \quad t \in [0,T].
  \end{equation*}
  Then, we choose $C_5 \in \bN$ sufficiently large, such that, for all $k \geq C_5$,
  \begin{equation*}
    \bP\Big(g\big(\|X^k-X\|_\infty, W_\eta(\L(X),\L(X^k))\big) \geq \frac{\epsilon}{3M}\Big) \leq \frac{\theta}{2}.
  \end{equation*}
  Setting $\bar C := \max_i \{C_i\}$, we get
  \begin{align*}
    &\bP(\|G^k-G\|_\infty \geq \epsilon)\\
    &\quad \leq \bP(\{\|G^k-G\|_\infty \geq \epsilon\} \cap \{\|X\|_\infty \vee \|X^k\|_\infty \leq N\}) + \bP(\{\|X\|_\infty \vee \|X^k\|_\infty \geq N\})\\
    &\quad \leq \bP(\{g(\|X^k-X\|_\infty, W_\eta(\L(X),\L(X^k))) \geq \frac{\epsilon}{3M}\} \cap \{\|X\|_\infty \vee \|X^k\|_\infty \leq N\} ) \\
    &\quad \qquad + \bP(\{\|X\|_\infty \vee \|X^k\|_\infty \geq N)\\
    &\quad \leq \theta
  \end{align*}
  for all $k \geq \bar{C}$.
\end{proof}

\begin{lemma}\label{lemma:app_space_change}
  Let $T>0$ and let kernels $K_b$, $K_\sigma$ satisfy Assumption~\ref{ass:kernel}. Let $A$ and $M$ be continuous processes such that
  \begin{itemize}
    \item $A$ is a finite variation process satisfying $A_t = A_0 + \int_0^t \alpha_s \dd s$ with $\bE[\int_0^T |\alpha_s|^q \dd s]<\infty$ for some $q \geq 2 + \frac{4}{\epsilon}$, and
	\item $M$ is a continuous local martingale satisfying $M_t = \int_0^t \beta_s \dd B_s$ for some Brownian motion $B$ with $\bE[\int_0^T |\beta_s|^p \dd s]<\infty$
  \end{itemize}
  and define $X$ by
  \begin{equation*}
    X_t := A_0 + \int_0^t K_b(s,t) \dd A_s + \int_0^t K_\sigma(s,t) \dd M_s, \qquad t \in [0,T].
  \end{equation*}
  Then
  \begin{itemize}
	\item there exists a measurable map
	\begin{equation*}
      \Phi \colon C([0,T];\bR^d) \times C([0,T];\bR^d) \to C([0,T];\bR^d)
	\end{equation*}
	s.t. $X = \Phi(A,M)$ a.s.
	\item If $(\hat{A},\hat{M},\hat{X}) \stackrel{\mathscr{D}}{\sim} (A,M,X)$ in $C([0,T];\bR^d \times \bR^d \times \bR^d)$, then $\hat{X}= \Phi(\hat{A},\hat{M})$ a.s.
  \end{itemize}
\end{lemma}

\begin{proof}
  By Assumption~\ref{ass:kernel} and Remark~\ref{rem:asskernel_to_assshort} there exist sequences $(K_b^n)_{n \in \bN}$ and $(K_\sigma^n)_{n \in \bN}$ of simple kernels $K_b^n, K_\sigma^n \colon [0,T]^2 \to \bR$, such that
  \begin{align}\label{eq:app_conv_K}
    \begin{split}
    \sup_{t \in [0,T]} \int_0^t |K^n_b(s,t)-K_b(s,t)|\dd s &\to 0,\\
	\sup_{t \in [0,T]} \int_0^t |K^n_\sigma(s,t)-K_\sigma(s,t)|^{2+\epsilon} \dd s &\to 0,
	\end{split}
  \end{align}
  as $n \to \infty$. For $(a,m) \in C([0,T]; \bR^d \times \bR^d)$ we define $\Phi_n(a,m)$ by
  \begin{equation*}
    \Phi_n(a,m)_t = a_0 + \int_0^t K_b^n(s,t) \dd a_s + \int_0^t K_\sigma^n(s,t) \dd m_s, \qquad t \in [0,T].
  \end{equation*}
  For each fixed $n \in \bN$ there are partitions $t_0=0 < t_1 < \ldots < t_m=T$ and $u_0=0 < u_1 < \ldots < u_m=T$ and $c_{ij} \in \bR$, $i,j \in \{0,\ldots,m-1\}$, such that
  \begin{equation*}
    K_\sigma^n(s,t) = \sum_{i,j=0}^{m-1} c_{ij} \mathbbm{1}_{(t_i,t_{i+1}]}(s) \mathbbm{1}_{(u_j,u_{j+1}]}(t), \qquad (s,t) \in \Delta_T.
  \end{equation*}
  Then
  \begin{equation*}
    \int_0^t K_\sigma^n(s,t) \dd m_s = \sum_{i,j=0}^{m-1} c_{ij} \mathbbm{1}_{(u_j,u_{j+1}]}(t) (m_{t \wedge t_{i+1}} - m_{t \wedge t_{i}})
  \end{equation*}
  and an equivalent relation holds for the finite variation part. Since the integrals reduce to finite sums of increments, $\Phi_n$ is measurable for all $n \in \bN$.

  We show
  \begin{equation}\label{eq:app_text}
    \Phi_n(A,M) \to X \text{ in } C([0,T];\bR^d) \text{ in probability.}
  \end{equation}
  The convergence of the FV part follows by the associativity of the Stieltjes integral and \eqref{eq:app_conv_K}. For the convergence of the martingale part one can mimic the steps below \eqref{eq:decomp_conv1}.

  Now define
  \begin{align*}
    \Phi(a,m) = \begin{cases}\lim_{n \to \infty} \Phi_n(a,m),& \text{if the limit exists in } C([0,T];\bR^d),\\
	0, &\text{else.}\end{cases}
  \end{align*}
  Since convergence in a Polish space (here $C([0,T];\bR^d)$) is a Borel property and each $\Phi_n$ is measurable, the map $\Phi$ itself is Borel measurable. By \eqref{eq:app_text}, for every admissible pair $(A,M)$, $\Phi(A,M)=X$ holds a.s. This finishes the proof of the first statement.

  Now define $G:=\{(a,m,x)\colon x= \Phi(a,m)\} \subseteq C([0,T];\bR^d \times \bR^d \times \bR^d)$. Since $\Phi$ is measurable, $G$ is a measurable set. By assumption $\bP((A,M,X) \in G)=1$. Then, by $(\hat{A},\hat{M},\hat{X}) \stackrel{\mathscr{D}}{\sim} (A,M,X)$ and assuming that $\hat{A},\hat{M},\hat{X}$ live on a probability space $(\hat{\Omega},\hat{\mathcal{F}},\hat{\bP})$, we finally have $\hat{\bP}((\hat{A},\hat{M},\hat{X}) \in G)=1$, i.e. $\hat{X}=\Phi(\hat{A},\hat{M})$ $\hat{\bP}$-a.s. This shows the second statement.
\end{proof}


\providecommand{\bysame}{\leavevmode\hbox to3em{\hrulefill}\thinspace}
\providecommand{\MR}{\relax\ifhmode\unskip\space\fi MR }
\providecommand{\MRhref}[2]{%
  \href{http://www.ams.org/mathscinet-getitem?mr=#1}{#2}
}
\providecommand{\href}[2]{#2}

\end{document}